\documentclass[11pt,reqno]{amsart}

\usepackage{preamble_general}
\usepackage{graphicx}
\usepackage{float}
\usepackage[
    colorlinks,
    linkcolor={red!50!black},
    citecolor={green!50!black},
    urlcolor={blue!70!black},
    linktocpage=true
]{hyperref}
\usepackage[nameinlink,noabbrev]{cleveref}
\usepackage{orcidlink}
\usepackage{preamble_abbreviations}

\theoremstyle{plain}
\newtheorem{maintheorem}[alphaprop]{Theorem}
\crefname{maintheorem}{theorem}{theorems}
\Crefname{maintheorem}{Theorem}{Theorems}
\crefname{alphaprop}{theorem}{theorems}
\Crefname{alphaprop}{Theorem}{Theorems}
\crefname{conjecture}{conjecture}{conjectures}
\Crefname{conjecture}{Conjecture}{Conjectures}
\crefname{question}{question}{questions}
\Crefname{question}{Question}{Questions}

\title{Dynamics on projective K3 surfaces: Herman rings}


\author{Junyu Cao \orcidlink{0009-0003-6875-3981}}
\date{\today}

\begin{document}

\begin{abstract}
We construct an automorphism of a smooth projective K3 surface with
positive topological entropy and a non-empty Fatou set.
The Fatou set contains invariant domains biholomorphic to products of
two annuli. On each domain, the automorphism is holomorphically conjugate
to a fixed-point-free irrational rotation.
These domains are two-dimensional analogues of \emph{Herman rings}.
The construction answers a question posed by McMullen in 2002 and several
questions posed by Cantat in his 2018 ICM address.
\end{abstract}

\maketitle
\setcounter{tocdepth}{2}
\tableofcontents
\enlargethispage{2pt}

\section{Introduction}

Let \(X\) be a complex K3 surface and let
\(T \colon X \to X\) be an automorphism. The \emph{(two-sided) Fatou set}
of \(T\) is
\[
\Fat(T)\coloneqq
\left\{
x \in X :
\{T^n : n \in \Z\}\text{ is equicontinuous on some open neighborhood of }x
\right\}.
\]
The set \(\Fat(T)\) is open and \(T\)-invariant, and it describes the
locally regular part of the dynamics in both forward and backward time.
For K3 automorphisms, this set agrees with the usual forward Fatou
set.\footnote{The automorphism \(T\) preserves a smooth volume form.
By Poincar\'e recurrence, some forward iterates \(T^{n_j}\) converge to
the identity locally uniformly on the forward Fatou set. The backward
iterates are then limits of forward iterates, so they are equicontinuous
there as well.}

The Fatou set records regular dynamics, whereas positive topological entropy
records chaotic dynamics. It is therefore natural to ask whether a
positive-entropy automorphism can have a non-empty Fatou set.
Cantat raised this question in his thesis \cite{CantatThesis1999}.
Independently, McMullen gave the first such examples
\cite{McMullen2002}*{Theorem~1.1} on non-projective K3 surfaces. In his
examples, there exists a \(T\)-invariant open neighborhood \(U\) of a fixed
point that is biholomorphic to a bidisk, and the restriction of the
automorphism to \(U\) is conjugate to a Diophantine rotation.
Thus \(U \subset \Fat(T)\). Such a domain is
called a \emph{Siegel disk}.
More recently, Iwasaki--Takada
\cites{IwasakiTakada2022,IwasakiTakada2023} and Iwasaki
\cite{Iwasaki2025} extended McMullen's non-projective constructions.

McMullen also proved that an automorphism of a projective K3 surface cannot
have a non-resonant Siegel disk \cite{McMullen2002}*{Theorem~7.2}. This does not
force the Fatou set to be empty: it could still contain a neighborhood of a
resonant fixed point or a fixed-point-free invariant domain. McMullen
therefore asked whether a positive-entropy automorphism of a projective K3
surface can have a non-empty Fatou set
\cite{McMullen2002}*{Introduction, Questions}.
Cantat later reiterated this question in his ICM address
\cite{Cantat2018}*{Questions~3.4}.

In this paper, we pursue the second possibility by constructing a
fixed-point-free invariant domain for a positive-entropy automorphism of a
projective K3 surface and proving that this domain lies in the Fatou set.

The basic model for this fixed-point-free approach is an invariant
\emph{Herman ring}, an annulus in one-dimensional complex dynamics on which
the map is holomorphically conjugate to an irrational rotation. Moncet
developed a two-dimensional analogue on a real rational surface. He
constructed a birational diffeomorphism whose entire real locus lies in the
Fatou set, and asked whether the same phenomenon can occur for an automorphism
\cite{Moncet2013}*{Theorem~A and Introduction, Open Question~2}. More
recently, Cantat suggested that higher-dimensional analogues of Herman rings
might be used to construct an automorphism whose Fatou set contains the
entire real locus of a real projective K3 surface
\cite{CantatFatou}*{Section~7}.

The main result of this paper answers the questions above affirmatively and
realizes this strategy.

\begin{maintheorem}\label{thm:complex-main}
There exist a complex projective K3 surface \(X\) and an automorphism
\(T \colon X \to X\) with the following properties.
The automorphism \(T\) has positive topological entropy and a non-empty
Fatou set. More precisely, \(\Fat(T)\) contains a \(T\)-invariant biannulus.
On this domain, \(T\) is holomorphically conjugate to a Diophantine rotation.
\end{maintheorem}

A \emph{biannulus} is a complex domain biholomorphic to a product of two
annuli. The invariant biannulus in \Cref{thm:complex-main} is a
two-dimensional analogue of a Herman ring.

\Cref{thm:complex-main} follows from a more precise statement about a real
K3 surface. To formulate it, we fix notation for rotations and make the
Diophantine condition precise.
Set \(\T\coloneqq\R/2\pi\Z\) and
\(\S^1\coloneqq\{z\in\C:\abs z=1\}\).
For \(\bu=(u_1,u_2)\in\R^2\), write \([\bu]\in\T^2\) for its class modulo
\((2\pi\Z)^2\).
We identify \(\T^2\) with \((\S^1)^2\) by
\([\bu]\mapsto(e^{iu_1},e^{iu_2})\).

For \(\bbeta=(\beta_1,\beta_2)\in\R^2\), the rotation
\(\Rot_{\bbeta}\colon\T^2\to\T^2\) is given by
\([\bu]\mapsto[\bu+\bbeta]\).
We call \(\bbeta\) its \emph{rotation vector}.
It is determined by the rotation modulo \((2\pi\Z)^2\).
The same vector defines a rotation of \((\C^*)^2\):
\[
R_{\bbeta}(z_1,z_2)
\coloneqq(e^{i\beta_1}z_1,e^{i\beta_2}z_2),
\qquad \text{for } (z_1,z_2)\in(\C^*)^2.
\]
The identification above conjugates \(\Rot_{\bbeta}\) to
\(R_{\bbeta}|_{(\S^1)^2}\).

The rotations in our main theorems have Diophantine rotation vectors.
\begin{definition}[Diophantine condition]
\label[definition]{dfn:diophantine-condition}
Let \(\gamma>0\) and \(\tau\geq2\).
The set \(\operatorname{DC}(\gamma,\tau)\subset\R^2\) consists of the
vectors \(\balpha=(\alpha_1,\alpha_2)\) satisfying the following bound.
For every \(\bfk=(k_1,k_2)\in\Z^2\setminus\{\bzero\}\), we have
\begin{equation}\label{eq:diophantine}
\operatorname{dist}(k_1\alpha_1+k_2\alpha_2,2\pi\Z)
\geq\gamma\,\norm{\bfk}_\infty^{-\tau}.
\end{equation}
Here \(\norm{\bfk}_\infty=\max\{\abs{k_1},\abs{k_2}\}\).
We call \(\balpha\) \emph{Diophantine} if it belongs to
\(\operatorname{DC}(\gamma,\tau)\) for some \(\gamma>0\) and \(\tau\geq2\).
The condition depends only on \([\balpha]\in\T^2\).
We also write \(\operatorname{DC}(\gamma,\tau)\) for its image in \(\T^2\).
\end{definition}

For \(r>0\), we define the biannulus
\[
\cA_r\coloneqq
\{(z_1,z_2)\in(\C^*)^2:e^{-r}<\abs{z_k}<e^r,\ k=1,2\}.
\]

The pair \((X,T)\) in \Cref{thm:complex-main} is obtained by complexifying
the following smooth real projective K3 surface and its \(\R\)-automorphism.

\begin{maintheorem}
\label{thm:real-main}
There exist a smooth projective surface \(X_{\R}\) over \(\R\) and an
\(\R\)-automorphism
\(T_{\R} \in \operatorname{Aut}_{\R}(X_{\R})\) with the following
properties. Write \(X\) and \(T\) for their complexifications.
Then \(X\) is a complex projective K3 surface and
\(T \in \operatorname{Aut}(X)\).
The Picard number satisfies \(\rho(X)\geq4\).
\begin{enumerate}[label=\textup{(\roman*)}]
\item\textup{\textbf{Real dynamics.}}
The real locus is the disjoint union of two tori:
\[
X_{\R}(\R) = M_0 \sqcup M_1,
\qquad M_\ell \simeq (\S^1)^2 \quad \text{for } \ell=0,1.
\]
There are a constant \(\gamma>0\) and a vector
\(\balpha=(\alpha_1,\alpha_2)\in\operatorname{DC}(\gamma,3)\).
For each \(\ell = 0,1\), there is a
real-analytic diffeomorphism
\(j_{\ell,\R} \colon (\S^1)^2 \longrightarrow M_\ell\)
that conjugates the real dynamics to a rotation:
\[
T_{\R} \circ j_{\ell,\R}
= j_{\ell,\R} \circ R_{(-1)^\ell\balpha},
\qquad \text{for } \ell = 0,1.
\]
Thus \(T_{\R}\) is conjugate to a rotation on each
component of the real locus. In particular,
\[
h_{\mathrm{top}}\bigl(T_{\R}|_{X_{\R}(\R)}\bigr) = 0.
\]
\item\textup{\textbf{Complex dynamics.}} The topological entropy of the holomorphic dynamical system \((X,T)\)
is
\[
h_{\mathrm{top}}(T)
= \log(7 + 4\sqrt3).
\]
\item\textup{\textbf{Invariant biannuli.}}
For some \(r>0\), each map \(j_{\ell,\R}\) extends to a holomorphic embedding
\(j_\ell \colon \cA_r \longrightarrow X\).
The extension satisfies \(j_\ell|_{(\S^1)^2} = j_{\ell,\R}\) and
\[
T \circ j_\ell = j_\ell \circ R_{(-1)^\ell\balpha}.
\]
The open sets \(W_\ell \coloneqq j_\ell(\cA_r)\) are disjoint.
They satisfy
\[
X_{\R}(\R) \subset W_0 \cup W_1 \subset \Fat(T).
\]
\end{enumerate}
\end{maintheorem}

We first describe the idea of the construction. A more detailed
\hyperref[par:proof-sketch]{sketch of the proof} appears later in this
introduction.

The surface \(X\) belongs to an algebraic family of K3 surfaces with a
natural automorphism. For a general member of this family, the automorphism
has positive entropy. The main task is therefore to control the
real dynamics.
For this, we use a degeneration.
The real dynamics of a singular fiber controls that of nearby smooth fibers.
This strategy appeared in Moncet's work \cite{Moncet2012}*{Section~5.3}
and in the work of Filip--Tosatti \cite{FilipTosatti2024}*{Section~2.4}.

The real loci of our singular fibers consist of two tori, on which the natural dynamics acts by
explicit rotations.
The two tori persist in nearby smooth fibers, where the dynamics on each
torus is close to a rotation.
By Herman's KAM theorem, this dynamics is conjugate to a fixed Diophantine
rotation up to a correction term in \(\R^2\).
Once we make this term vanish, we obtain real-analytic conjugacies to
rotations on the real locus.
They extend holomorphically and give the invariant biannuli in
property~\textup{(iii)}.
The extension argument also appeared in
\cite{Moncet2013}*{Proposition~1.2}.

A \emph{Fatou component} is a connected component of \(\Fat(T)\).
The two invariant biannuli in \Cref{thm:real-main} are disjoint, but this
alone does not separate the Fatou components containing them. The next
result shows that the tori \(M_0\) and \(M_1\) lie in distinct Fatou
components. It also records how an
involution couples their dynamics.

\begin{proposition}\label[proposition]{prop:fatou-components-and-coupling}
For the pair \((X,T)\) constructed in \Cref{thm:real-main}, let
\(U_\ell\) be the Fatou component containing \(M_\ell\), for \(\ell=0,1\).
Then \(U_0\ne U_1\). In particular, \(\Fat(T)\) is disconnected.

There is an involution \(A\in\operatorname{Aut}(X)\), defined over
\(\R\), such that
\[
A(M_0)=M_1,\qquad A(U_0)=U_1,\qquad
A\circ T\circ A=T^{-1}.
\]
The restriction \(A|_{M_0}\) conjugates \(T|_{M_0}\) to
\(T^{-1}|_{M_1}\).
\end{proposition}

The proof is given at the end of \Cref{sec:proof-main}.

\begin{remark}
\Cref{thm:real-main} also shows that the real locus is non-empty and
lies entirely in the Fatou set. Although the complex automorphism has
positive entropy, its restriction to the real locus has entropy zero.
These two properties answer two questions of Cantat affirmatively
\cite{Cantat2018}*{Questions~3.5(2) and (3)}.
\end{remark}

\begin{remark}
The disjoint \(T\)-invariant open sets \(W_0\) and \(W_1\) in
\Cref{thm:real-main} imply that no \(T\)-orbit is dense.
On each \(W_\ell\), the automorphism \(T\) is conjugate to a rotation
without periodic points. Hence all periodic points lie in the complement of
\(W_0 \cup W_1\), and the set of periodic points is not dense either.
\end{remark}

For \(0 \ne \Omega \in \rH^0(X,K_X)\), the \emph{canonical volume} of
\(X\) is the probability measure
\begin{equation}\label{eq:canonical-volume}
\vol_X \coloneqq
\frac{\Omega \wedge \overline\Omega}
{\int_X \Omega \wedge \overline\Omega}.
\end{equation}
It is independent of the choice of \(\Omega\) and invariant under every
automorphism of \(X\).

\begin{corollary}\label[corollary]{cor:canonical-volume-nonergodic}
For the pair \((X,T)\) obtained in \Cref{thm:real-main}, the canonical
volume \(\vol_X\) is not ergodic with respect to \(T\). More precisely,
the disjoint \(T\)-invariant open sets \(W_0\) and \(W_1\) have positive
canonical volume, so
\[
0 < \vol_X(W_\ell) < 1
\qquad \text{for } \ell=0,1.
\]
\end{corollary}

Next, we relate the Fatou set to the canonical currents of \(T\).
Let \(\lambda_1\coloneqq\lambda_1(T)>1\) be the first dynamical degree of
\((X,T)\).
Recall from \cite{Cantat2001}*{Proposition~3.4} that the Green currents of
\(T\) are the weak limits
\[
\eta_\pm\coloneqq\lim_{n\to\infty}
\lambda_1^{-n}(T^{\pm n})^*\omega.
\]
Here \(\omega\) is a suitable K\"ahler form.
On \(\Fat(T)\), equicontinuity and Cauchy estimates give locally uniform
bounds for the pullbacks in this formula. Thus both Green currents vanish
on \(\Fat(T)\). So does the measure of maximal entropy
\(\mu_T\coloneqq\eta_+\wedge\eta_-\); see
\cite{Cantat2001}*{Theorem~6.2} and
\cite{DinhSibony2005}*{Sections~3--5}.

\begin{corollary}\label[corollary]{cor:maximal-entropy-support}
For the pair \((X,T)\) obtained in \Cref{thm:real-main}, the support of
the measure of maximal entropy \(\mu_T\) does not have full Lebesgue measure.
More precisely, the open sets \(W_0\) and \(W_1\) satisfy
\[
\supp(\mu_T)\cap(W_0\cup W_1)=\varnothing.
\]
\end{corollary}

Tosatti conjectured \cite{Tosatti2021}*{Conjecture~7.3} that the support of
the measure of maximal entropy of every positive-entropy automorphism of a
projective K3 surface has full Lebesgue measure. Filip--Tosatti later
constructed canonical currents on projective K3 surfaces that are not fully
supported, while leaving open whether the measure of maximal entropy is fully
supported \cite{FilipTosatti2024}*{Introduction and Theorems~1--2}.
\Cref{cor:maximal-entropy-support} disproves Tosatti's conjecture.

\vspace{1em}
\noindent\phantomsection\label{par:proof-sketch}\textbf{Sketch of the proof.}
We now implement the idea above, using the notation of the later sections.
We consider a family of complete intersections in
\(P_\R\coloneqq(\P^1_\R)^4\), with coordinates \((x,y,z,w)\).
Set \(L_\R\coloneqq\cO_{P_\R}(1,1,1,1)\).
The parameter space for this family is
\(\cV\coloneqq\rH^0(P_\R,L_\R)^{\oplus2}\),
whose elements are ordered pairs of sections of \(L_\R\).
For \(s\in\cV\), let \(X_{s,\R}\subset P_\R\) be the common zero set of
the two sections. It is a complete intersection of two real divisors of
multidegree \((1,1,1,1)\) (\textsection~\ref{sec:real-family}).
Write
\(X_s\coloneqq X_{s,\R}\times_{\R}\C\).
For parameters \(s\) in a nonempty Zariski open set \(\cG\), called the
\emph{good locus}, the surfaces \(X_{s,\R}\) are smooth real projective
K3 surfaces.
For these parameters, the two coordinate projections onto the
\((y,w)\) and \((x,z)\) factors are finite flat double covers.
Denote their deck involutions by \(A_{s,\R}\) and
\(B_{s,\R}\), respectively, and set
\[
T_{s,\R}\coloneqq B_{s,\R}\circ A_{s,\R},
\qquad
T_s\coloneqq T_{s,\R}\times_{\R}\C.
\]
The action on divisor classes gives the topological entropy
\(h_{\mathrm{top}}(T_s) = \log(7 + 4\sqrt3)\) for every \(s\in\cG\)
(\textsection~\ref{sec:complex-dynamics}).
It remains to find a good parameter whose real locus consists of tori on
which \(T_{s,\R}\) is conjugate to Diophantine rotations.

We find such parameters near singular surfaces whose real dynamics is
explicit. Identify \(\P^1(\R)\) with \(\T\) by the Cayley map
\(\theta\mapsto[\cos(\theta/2):\sin(\theta/2)]\).
This gives \(P_\R(\R)\simeq\T^2_{x,z}\times\T^2_{y,w}\), where
\(\T^2_{x,z}\coloneqq\T_x\times\T_z\) and
\(\T^2_{y,w}\coloneqq\T_y\times\T_w\).
For \(\bt=(t_1,t_2)\in\R^2\) with \(t_1t_2\ne0\), we construct reducible
surfaces \(X_{c(\bt),\R}\), which we call the \emph{singular rotation
models}. Here \(c\colon\R^2\to\cV\) is an affine map.
The real locus of each model consists of the graphs of \(\id\) and
\(\Rot_{-\nu(\bt)}\) from \(\T^2_{x,z}\) to \(\T^2_{y,w}\), where
\[
\nu(\bt)\coloneqq(2\arctan t_1,2\arctan t_2).
\]
Both graphs lie in the smooth locus of the model.
We parametrize both graphs by their \((x,z)\)-coordinates.
Over each point of \(\T^2_{y,w}\) or \(\T^2_{x,z}\), the corresponding
projection from the real locus has exactly two preimages, one on each
graph.
The \emph{sheet exchange} of the projection swaps these two points.
Apply the \(yw\)-exchange first and then the \(xz\)-exchange, as in
\(T_{s,\R}=B_{s,\R}\circ A_{s,\R}\).
This composition acts on the two graphs by rotations with opposite vectors
\(\nu(\bt)\) and \(-\nu(\bt)\)
(\textsection~\ref{sec:singular-rotation-models}).

We smooth these models by perturbing the defining sections in a fixed
direction \(v_0\in\cV\setminus\{0\}\). The perturbed parameters have the
form
\[
s(\bt,\lambda)\coloneqq c(\bt)+\lambda v_0.
\]
Here \(\bt\) adjusts the rotation, and \(\lambda\in\R\) controls the
smoothing. We choose \(v_0\), an open disk \(D_0\subset\R^2\), and
\(\delta_0>0\) so that \(s(\bt,\lambda)\in\cG\) for every
\(\bt\in\overline{D_0}\) and every
\(\lambda\in(-\delta_0,\delta_0)\setminus\{0\}\).
This uniformity in \(\bt\) is essential. The parameter \(\bt\) will be
adjusted after \(\lambda\) is fixed, and the adjusted parameter must remain
good.
Fix a Diophantine vector \(\balpha\in\nu(D_0)\), and set
\(\bt_0\coloneqq\nu^{-1}(\balpha)\).

Choose \(0<\delta<\delta_0\) sufficiently small.
For \(\bt\) near \(\bt_0\) and \(\lambda\in(-\delta,\delta)\), the
real-analytic implicit function theorem represents the real locus
\(X_{s(\bt,\lambda),\R}(\R)\) as the disjoint union of the graphs of
real-analytic diffeomorphisms \(g_0(\bt,\lambda)\) and
\(g_1(\bt,\lambda)\) from \(\T^2_{x,z}\) to \(\T^2_{y,w}\)
(\Cref{lem:graph-persistence}).
These maps depend real-analytically on the parameters and the torus
variable, and
\[
g_0(\bt,0)=\id,
\qquad g_1(\bt,0)=\Rot_{-\nu(\bt)}.
\]
The \(yw\)-exchange preserves the \((y,w)\)-coordinates, and the
\(xz\)-exchange preserves the \((x,z)\)-coordinates.
Thus their composition acts on the first graph, in its
\((x,z)\)-coordinates, as
\[
f_{\bt,\lambda}=g_1(\bt,\lambda)^{-1}\circ g_0(\bt,\lambda).
\]
For \(\lambda\ne0\), the parameter \(s=s(\bt,\lambda)\) is good.
The deck involutions \(A_{s,\R}\) and \(B_{s,\R}\) then restrict to the two
sheet exchanges. Hence \(f_{\bt,\lambda}\) represents the restriction of
\(T_{s,\R}\) to the first component.
At \(\lambda=0\), it equals \(\Rot_{\nu(\bt)}\).
Thus \(f_{\bt,\lambda}\) is a small real-analytic perturbation of
\(\Rot_{\nu(\bt)}\).

A small smoothing alone does not give the desired conjugacy. A perturbation
of the Diophantine rotation \(\Rot_{\balpha}\) need not be conjugate to
\(\Rot_{\balpha}\). Herman's KAM theorem writes it instead as
\(\Rot_{\btheta}\circ\psi\circ\Rot_{\balpha}\circ\psi^{-1}\) for a
real-analytic diffeomorphism \(\psi\) and a small vector \(\btheta\)
(\textsection~\ref{sec:herman-counterterm}).
The vector \(\btheta\) is the \emph{correction term}.
It lies in \(\R^2\), so its vanishing imposes two conditions.
We meet them with the two parameters in \(\bt\).
After shrinking the neighborhood of \(\bt_0\) and decreasing \(\delta\),
write \(\Theta(\bt,\lambda)\in\R^2\) for the correction term of
\(f_{\bt,\lambda}\). It depends continuously on \((\bt,\lambda)\).
At \(\lambda=0\), the map \(f_{\bt,0}=\Rot_{\nu(\bt)}\) is a rotation.
For a rotation near \(\Rot_{\balpha}\), the correction term is the
difference of the rotation vectors (\Cref{thm:herman}\textup{(ii)}).
Hence
\[
\Theta(\bt,0)=\nu(\bt)-\balpha.
\]
The map \(\nu\) is an orientation-preserving diffeomorphism.
Thus \(\Theta(\,\cdot\,,0)\) has topological degree \(1\) at zero on a small disk
centered at \(\bt_0\).
By homotopy invariance of the degree, after decreasing \(\delta\), for every
\(\lambda\in(-\delta,\delta)\setminus\{0\}\) there is a parameter
\(\bt_\lambda\) at which the correction term vanishes.
The uniform choice of smoothing ensures \(s\coloneqq s(\bt_\lambda,\lambda)\in\cG\).
Herman's theorem then conjugates \(T_{s,\R}\) to \(\Rot_{\balpha}\) on the
first torus.
The deck involution \(A_{s,\R}\) exchanges the tori and conjugates
\(T_{s,\R}\) to its inverse.
This gives the opposite rotation on the second torus and proves
property~\textup{(i)} of \Cref{thm:real-main}
(\textsection~\ref{sec:smoothing-kam}). Property~\textup{(ii)} and the
bound \(\rho(X_s)\geq4\) hold for every \(s\in\cG\).
Holomorphic extensions of the real-analytic conjugacies yield the
invariant biannuli in property~\textup{(iii)}
(\textsection~\ref{sec:torus-complexification}).

The algebraic family and its natural automorphisms were studied by
Bhargava--Ho--Kumar \cite{BhargavaHoKumar2016}*{Section~7} and
Hashimoto--Oda \cite{HashimotoOda2026}*{Section~5}.
Our contribution is the singular rotation model and a smoothing that
preserves a rotation with a prescribed Diophantine vector.

The paper first collects the analytic tools and then follows the
construction. \Cref{sec:kam-conjugacy} states these tools: Herman's normal form and the
complexification of an invariant rotation torus. \Cref{sec:algebraic}
constructs the K3 family and its deck involutions and computes the entropy.
\Cref{sec:smoothing} builds the singular rotation models and carries out
the smoothing with prescribed rotation. \Cref{sec:proof-main} assembles
these ingredients and separates the Fatou components.
\Cref{sec:metric-entropy-conjecture} discusses a conjecture on metric
entropy. \Cref{app:algebraic-details} describes the good locus.
\Cref{fig:complex-orbits} illustrates the dynamics with four orbits computed
in a numerical model.
\begin{figure}[H]
\centering
\includegraphics[width=0.8\textwidth]{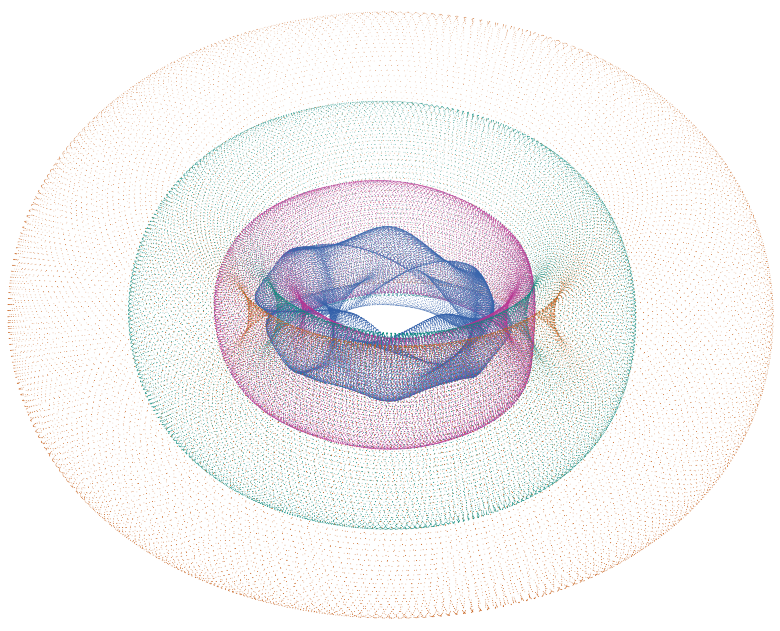}
\caption[Four numerically computed orbits.]{Four orbits in a numerical model
of \((X,T)\)\protect\footnotemark, each computed
for \(32{,}768\) forward iterations. Each color represents one orbit.}
\label{fig:complex-orbits}
\end{figure}
\footnotetext{In affine coordinates \((x,y,z,w)\), put
\(u=(1,y,x,xy)^{\mathsf t}\) and \(v=(1,w,z,zw)^{\mathsf t}\).
The numerical surface in \((\P^1)^4\) is defined by multihomogenizing
the following equations:
\[
\begin{aligned}
(y-x)\bigl(w-z+b(1+zw)\bigr)+10^{-3}u^{\mathsf t}M_1v&=0,\\
\bigl(y-x+a(1+xy)\bigr)(w-z)+10^{-3}u^{\mathsf t}M_2v&=0.
\end{aligned}
\]
The coefficient matrices are
\[
M_1=\begin{pmatrix}
-3&-7&3&-3\\
-5&0&8&-5\\
-7&-4&4&-3\\
5&5&-5&1
\end{pmatrix},\qquad
M_2=\begin{pmatrix}
-8&-4&-4&-4\\
-7&3&5&1\\
-9&1&3&3\\
8&4&4&4
\end{pmatrix}.
\]
The fixed rational parameters, rounded here, are
\(a\approx0.52228966285143913\) and
\(b\approx0.90837855261203148\).}

\vspace{1em}
\noindent\textbf{Notation and conventions.}
For a real variety \(Y_{\R}\), line bundle \(L_{\R}\), and morphism
\(f_{\R}\), we write \(Y=Y_{\R}\times_{\R}\C\), \(L=L_{\R}\times_{\R}\C\),
and \(f=f_{\R}\times_{\R}\C\) for their complexifications. The set
\(Y_{\R}(\R)\) of real points is called the \emph{real locus} and is
naturally viewed as a subset of \(Y(\C)\). A real morphism induces a map
between real loci, which we denote by the same symbol. This map agrees
with the restriction of its complexification.

For sections \(F_1,\ldots,F_r\) of line bundles on a scheme \(Y\),
\(V_Y(F_1,\ldots,F_r)\) denotes their common zero subscheme, with \(Y\)
omitted when the ambient space is clear.

Write \(\operatorname{Diff}^{\omega}_0(\T^2)\) for the group of
real-analytic diffeomorphisms isotopic to the identity, with the relative
\(C^q\)-topology whenever an integer \(q\geq0\) is specified.

For a bounded open set \(D\subset\R^n\), a continuous map
\(F\colon\overline D\to\R^n\), and a point \(y\notin F(\partial D)\),
\(\deg(F,D,y)\in\Z\) denotes the topological (Brouwer) degree
\cite{Deimling1985}*{Chapter~1}.
If \(F\) is \(C^1\) and \(y\) is a regular value, it is the sum of the
signs of the Jacobian determinant of \(F\) over \(F^{-1}(y)\).
It is invariant under homotopies that avoid \(y\) on \(\partial D\).
If \(\deg(F,D,y)\neq0\), then \(y\in F(D)\).

\vspace{1em}
\noindent\textbf{Use of generative AI.}
An initial version of the underlying idea was suggested by the
\href{https://github.com/frenzymath/Rethlas}{Rethlas agent}, using the
\texttt{gpt-5.6-sol} model, and was subsequently developed and verified by
the author. The author also used
\href{https://openai.com/codex/}{OpenAI Codex} and
\href{https://claude.com/claude-code}{Claude Code} for grammatical
corrections and stylistic polishing, and \href{https://chatgpt.com/}{ChatGPT Pro} to
assist in locating potentially relevant references. All AI-generated
suggestions, including bibliographic suggestions, were treated as provisional
and independently verified by the author. The author assumes full
responsibility for every argument, citation, and the final text.

\vspace{1em}
\noindent\textbf{Acknowledgements.}
The author thanks his advisor, Valentino Tosatti, for suggesting the topic,
for reading a preliminary draft, and for his constant support.
The author also thanks Eric Bedford, Serge Cantat, Jeffrey Diller, and
Curtis McMullen for helpful feedback on a preliminary draft.
The author is especially grateful to Serge Cantat for detailed suggestions.
They greatly improved the exposition and led to
\Cref{prop:fatou-components-and-coupling}.

\section{Herman's normal form and complexification}\label{sec:kam-conjugacy}

This section provides the two analytic ingredients for the proof of
\Cref{thm:real-main}: Herman's normal form
(\Cref{sec:herman-counterterm}) and the complexification of rotation tori
(\Cref{sec:torus-complexification}).

\subsection{Diophantine vectors and Herman's normal form}
\label{sec:herman-counterterm}

We first show that Diophantine vectors with the fixed exponent \(3\) are dense.
\begin{lemma}[Density of Diophantine vectors with exponent \(3\)]
\label[lemma]{lem:diophantine-choice}
The union \(\bigcup_{\gamma>0}\operatorname{DC}(\gamma,3)\) has full
Haar measure in \(\T^2\).
In particular, every nonempty open subset of \(\R^2\) contains a vector
\(\balpha\in\operatorname{DC}(\gamma,3)\) for some \(\gamma>0\).
The constant \(\gamma\) may depend on the chosen vector.
\end{lemma}

\begin{proof}
The proof is standard; see \cite{Herman1979}*{Annex, Complement~(2.5), p.~204} for example.
\end{proof}

Herman's local KAM theorem applies to a Diophantine rotation.
It expresses a small perturbation as a conjugate of that rotation
followed by a small rotation.
The vector of the extra rotation is the \emph{correction term}.
We need its continuous dependence on the map and its value on nearby rotations.
The following formulation for maps of \(\T^2\) includes the
continuous choice provided by
\cite{Herman1979}*{Annex, Theorem~2.2, p.~203}.

\begin{theorem}[Herman's normal form with continuous correction term]
\label[theorem]{thm:herman}
\label[theorem]{prop:continuous-herman-counterterm}
Fix \(\gamma>0\), \(\tau\geq2\), and
\(\balpha\in\operatorname{DC}(\gamma,\tau)\). For every
\(\eta>0\), there are an integer \(q\geq1\) and a \(C^q\)-open neighborhood
\(\cU_{\balpha}\subset\operatorname{Diff}^{\omega}_0(\T^2)\) of
\(\Rot_{\balpha}\) with the following properties.
There is a continuous map for the \(C^q\)-topology
\[
\vartheta_{\balpha}\colon\cU_{\balpha}\longrightarrow(-\eta,\eta)^2.
\]
\begin{enumerate}[label=\textup{(\roman*)}]
\item\textup{\textbf{Normal form.}}
For every \(g\in\cU_{\balpha}\), there is a
\(\psi\in\operatorname{Diff}^{\omega}_0(\T^2)\) such that
\begin{equation}\label{eq:herman-normal-form}
g=\Rot_{\btheta}\circ\psi\circ\Rot_{\balpha}\circ\psi^{-1},
\qquad \btheta=\vartheta_{\balpha}(g).
\end{equation}
\item\textup{\textbf{Nearby rotations.}}
There is a radius \(r_{\balpha}\in(0,\pi)\) with the following property.
If \(\bbeta\in\R^2\) satisfies
\(\norm{\bbeta-\balpha}_\infty<r_{\balpha}\), then
\(\Rot_{\bbeta}\in\cU_{\balpha}\) and
\begin{equation}\label{eq:herman-rotation-counterterm}
\vartheta_{\balpha}(\Rot_{\bbeta})=\bbeta-\balpha.
\end{equation}
\end{enumerate}
\end{theorem}

\begin{proof}
Each \(g\) sufficiently \(C^0\)-close to \(\Rot_{\balpha}\) has a unique
real-analytic displacement
\(d_g\colon\T^2\to(-\pi,\pi)^2\) such that
\[
g([\bu])=[\bu+\balpha+d_g([\bu])].
\]
Define a lift of \(g\) by
\[
G_g(\bu)\coloneqq \bu+\balpha+d_g([\bu]).
\]
This lift depends continuously on \(g\) in the \(C^q\)-topology.
This is the local correspondence between lifts and torus maps described in
\cite{Herman1979}*{Annex, Section~1, p.~202}.

The vector \(\balpha/(2\pi)\) has Diophantine constant
\(\gamma/(2\pi)\) and exponent \(\tau\) in Herman's period-\(1\)
normalization. After rescaling to our period-\(2\pi\) coordinates,
\cite{Herman1979}*{Annex, Theorem~2.2, p.~203} applies to
\(G_g\) near \(\bu\mapsto \bu+\balpha\). Choose an integer \(q\) sufficiently
large. The continuity and analytic-regularity clauses give a
continuous correction term \(\btheta(g)\) and a real-analytic lift
\(\Psi_g\) of a diffeomorphism
\(\psi_g\in\operatorname{Diff}^{\omega}_0(\T^2)\), which satisfies
\begin{equation}\label{eq:herman-lifted-normal-form}
G_g\bigl(\Psi_g(\bu)\bigr)=\Psi_g(\bu+\balpha)+\btheta(g).
\end{equation}
Projecting \eqref{eq:herman-lifted-normal-form} to the torus gives
\eqref{eq:herman-normal-form}.

If \(g=\Rot_{\bbeta}\) with \(\bbeta\) close to \(\balpha\) in \(\R^2\),
then \(G_g(\bu)=\bu+\bbeta\). Iterating
\eqref{eq:herman-lifted-normal-form} gives
\[
\Psi_g(\bu+n\balpha)=\Psi_g(\bu)+n\bigl(\bbeta-\btheta(g)\bigr).
\]
Since \(\psi_g\) is isotopic to the identity, the displacement
\(\Psi_g-\id\) is periodic and hence bounded. Dividing the last
identity by \(n\) yields \(\btheta(g)=\bbeta-\balpha\).
In particular, \(\btheta(\Rot_{\balpha})=\bzero\). By continuity, we may shrink
the neighborhood \(\cU_{\balpha}\) so that
\(\btheta(g)\in(-\eta,\eta)^2\) throughout it. Choose \(r_{\balpha}\in(0,\pi)\)
small enough that all rotations with
\(\norm{\bbeta-\balpha}_\infty<r_{\balpha}\) lie in this neighborhood.
Taking \(\vartheta_{\balpha}=\btheta\) proves the theorem.
\end{proof}

\subsection{Complexification of an invariant rotation torus}
\label{sec:torus-complexification}

We now pass from a real-analytic rotation conjugacy on an invariant real
torus to a holomorphic conjugacy on a tubular neighborhood biholomorphic to
a biannulus \(\cA_r\).
Recall that \(\cA_r\) is the product of two copies of the annulus
\(e^{-r}<\abs{z}<e^r\), for \(r>0\).
We will apply this to each component of the real locus.
The extension relies on the following property of the tori in the real
locus.

A real submanifold \(M\) of a complex manifold \(Y\) is
\emph{maximally totally real} if its tangent spaces satisfy
\[
T_pY=T_pM\oplus iT_pM
\qquad \text{for } p\in M.
\]
The direct sum is taken over \(\R\).
In particular, \(\dim_{\R}M=\dim_{\C}Y\).

The following proposition slightly reformulates
\cite{Moncet2013}*{Proposition~1.2}.

\begin{proposition}[Complexification of an invariant rotation torus]
\label[proposition]{prop:herman-ring-complexify-torus}
Let \(Y\) be a complex surface and \(f \in \operatorname{Aut}(Y)\). Let
\(M \subset Y\) be a compact maximally totally real real-analytic torus.
For some \(\bbeta \in \R^2\), suppose that a real-analytic diffeomorphism
\(j_{\R} \colon (\S^1)^2 \to M\) satisfies
\[
f \circ j_{\R} = j_{\R} \circ R_{\bbeta}.
\]
Then \(j_{\R}\) extends to a holomorphic embedding
\(j \colon \cA_r \to Y\) for some \(r > 0\). This extension satisfies
\begin{equation}\label{eq:abstract-complex-conjugacy}
f \circ j = j \circ R_{\bbeta}.
\end{equation}
Set \(W \coloneqq j(\cA_r)\). Then \(W\) is invariant and
\begin{equation}\label{eq:all-integer-iterates}
f^n|_W = j \circ R_{n\bbeta} \circ j^{-1},
\qquad \text{for } n \in \Z.
\end{equation}
In particular, \(W \subset \Fat(f)\).
\end{proposition}

\begin{proof}
Regard \((\S^1)^2\) as the standard maximally totally real real-analytic
submanifold of \((\C^*)^2\).
We first extend \(j_{\R}\) to a biholomorphism
\(J\colon V\xrightarrow{\sim}V'\) between neighborhoods of \((\S^1)^2\)
and \(M\).
This extension is used without proof in
\cite{Moncet2013}*{Proposition~1.2}; we recall the standard argument.
Near each point, maximal total reality gives holomorphic coordinates in
which the torus is \(\R^2\subset\C^2\).
In such coordinates, the convergent power series of \(j_{\R}\) define
local holomorphic extensions.
By the identity theorem, these extensions are unique and agree near the
torus on overlaps, so they glue.
The same construction applies to \(j_{\R}^{-1}\).
The two compositions are the identity on the tori, hence near them.
Shrinking the neighborhoods gives \(J\).
Compactness gives an \(r>0\) such that
\(\overline{\cA_r}\subset V\). Set
\[
j\coloneqq J|_{\cA_r}\colon\cA_r\xrightarrow{\sim}W,
\qquad W\coloneqq J(\cA_r).
\]

The holomorphic maps \(f\circ j\) and \(j\circ R_{\bbeta}\) agree on
\((\S^1)^2\), so analytic continuation gives
\eqref{eq:abstract-complex-conjugacy} on \(\cA_r\). Since
\(R_{\bbeta}(\cA_r)=\cA_r\), the set \(W\) is invariant, and
iteration yields \eqref{eq:all-integer-iterates}.

The vectors \(\bigl(e^{in\beta_1},e^{in\beta_2}\bigr)\) lie in the compact
torus \((\S^1)^2\). Hence every sequence of integer rotations has a
subsequence converging uniformly on compact subsets of \(\cA_r\), and
conjugation by \(j\) gives the same conclusion for the corresponding
iterates of \(f\) on \(W\). By the Arzel\`a--Ascoli theorem, the family of
integer iterates is equicontinuous on \(W\), so \(W\subset\Fat(f)\).
\end{proof}

\section{The algebraic K3 construction}\label{sec:algebraic}
We now construct the family and its automorphisms described in the introduction.
We verify the K3 property and compute the action on divisor classes and
holomorphic two-forms.

These surfaces and their natural automorphisms were studied by
Bhargava--Ho--Kumar through the penteract construction
\cite{BhargavaHoKumar2016}*{Section~7}. More recently, Hashimoto--Oda
studied the same complete intersections under the name \emph{type II K3
surfaces}, identifying the Picard lattice of a very general member and
showing that its full automorphism group is generated by the six natural
deck involutions \cite{HashimotoOda2026}*{Section~5}.

\subsection{The real family and its involutions}
\label{sec:real-family}
Put
\[
P_{\R} \coloneqq (\P^1_{\R})^4
= \P^1_{x,\R} \times \P^1_{y,\R}
  \times \P^1_{z,\R} \times \P^1_{w,\R},
\qquad
L_{\R} \coloneqq \cO_{P_{\R}}(1,1,1,1).
\]
The real parameter space consists of ordered pairs of sections of
\(L_\R\):
\[
\cV\coloneqq \rH^0(P_{\R},L_{\R})^{\oplus2}\simeq\R^{32}.
\]
For an ordered pair
\(s=(F_{1,s},F_{2,s})\in\cV\), define the zero locus
\[
X_{s,\R}\coloneqq V_{P_{\R}}(F_{1,s},F_{2,s})\subset P_{\R}.
\]

Write
\[
P_{yw,\R}\coloneqq\P^1_{y,\R}\times\P^1_{w,\R},
\qquad
P_{xz,\R}\coloneqq\P^1_{x,\R}\times\P^1_{z,\R}.
\]
There are six coordinate projections onto pairs of factors.
We use the following two, whose pairs of coordinates are complementary:
\begin{equation}\label{eq:two-projections}
\pi_{yw}\colon X_{s,\R}\longrightarrow P_{yw,\R},
\qquad
\pi_{xz}\colon X_{s,\R}\longrightarrow P_{xz,\R}.
\end{equation}
The following proposition provides the parameters used later.
We prove it in \Cref{app:algebraic-details}.
\begin{proposition}\label[proposition]{prop:good-locus}
There is a nonempty real Zariski-open subset \(\cG\subset\cV\) such
that every \(s\in\cG\) satisfies the following conditions.
\begin{enumerate}[label=\textup{(\roman*)}]
\item The surface \(X_{s,\R}\) is smooth.
\item Both projection maps in \eqref{eq:two-projections}
are finite flat morphisms of degree two (double covers).
\end{enumerate}
\end{proposition}

We fix the set \(\cG\) defined in \eqref{eq:penteract-good-locus} and call
it the \emph{good locus}, as in the introduction.
Its definition also requires the auxiliary models in
\Cref{app:algebraic-details} to be smooth.
The proof there shows that all six coordinate projections are finite flat
double covers on \(\cG\).

The following lemma gives canonical
involutions for these double covers and explains their compatibility with
complexification.

\begin{lemma}[Canonical involution of a double cover]
\label[lemma]{lem:quadratic-cover-involution}
Let \(\pi\colon X\to Y\) be a finite flat morphism of degree two with
\(Y\) locally Noetherian. There is a canonical automorphism \(\tau_\pi\)
of \(X\) over \(Y\) satisfying \(\tau_\pi^2=\id\). Its formation commutes
with arbitrary base change. If \(\pi\) is étale, then \(\tau_\pi\) is its
deck involution.
\end{lemma}

\begin{proof}

The algebra \(\pi_*\cO_X\) is locally free of rank two.
Locally, choose a basis \(1,\xi\) with \(\xi^2=a\xi+b\).
The map \(\xi\mapsto a-\xi\) preserves this relation and is an involution.
Let \(\operatorname{Tr}_\pi(u)\) denote the trace of multiplication by \(u\).
The involution is given intrinsically by
\[
u\longmapsto\operatorname{Tr}_\pi(u)-u.
\]
Thus the local maps glue and commute with base change.
On an étale geometric fiber, this map exchanges the two factors of
\(\overline{k}\times\overline{k}\), so it is the deck involution.
\end{proof}
\begin{remark}
On each geometric fiber, the involution \(\tau_\pi\) exchanges the two
points when they are distinct and fixes the unique point otherwise.
\end{remark}

For \(s \in \cG\), let \(A_{s,\R}\) and \(B_{s,\R}\) be the canonical
involutions of \(\pi_{yw}\) and \(\pi_{xz}\), respectively,
given by \Cref{lem:quadratic-cover-involution}.
We call them the \emph{deck involutions}. Define
\[
T_{s,\R}\coloneqq B_{s,\R}\circ A_{s,\R}.
\]

\subsection{The complex K3 surfaces and their entropy}
\label{sec:complex-dynamics}

For \(s\in\cG\), the complexified surface is
\[
X_s=V_P(F_{1,s},F_{2,s})\subset P=(\P^1_{\C})^4,
\qquad L=\cO_P(1,1,1,1).
\]
The complexified projections, still denoted
\(\pi_{yw}\) and \(\pi_{xz}\), are finite flat double covers.
By the base-change compatibility in \Cref{lem:quadratic-cover-involution},
their deck involutions are \(A_s\) and \(B_s\), respectively,
and \(T_s=B_s\circ A_s\).

\begin{proposition}\label[proposition]{prop:k3-entropy}
For every \(s \in \cG\), \(X_s\) is a smooth projective complex K3
surface with Picard number at least \(4\). For every non-vanishing holomorphic
\(2\)-form \(\Omega_s \in \rH^0(X_s,K_{X_s})\), the deck involutions
and their composition satisfy
\[
A_s^*\Omega_s=B_s^*\Omega_s=-\Omega_s,
\qquad
T_s^*\Omega_s = \Omega_s.
\]
The algebraic automorphism \(T_s\) has entropy
\[
h_{\mathrm{top}}(T_s) = \log(7 + 4\sqrt3).
\]
\end{proposition}

\begin{proof}
We first verify the K3 property.
We then compute the intersection form on a \(T_s^*\)-invariant rank-four
sublattice of the N\'eron--Severi group.
Finally, we determine the involution actions and deduce the assertions for
\(T_s\).

\noindent\textbf{Step 1: the K3 property.}
Since \(s\in\cG\), the real surface \(X_{s,\R}\) is smooth, and hence so
is its complexification \(X_s\). Moreover, \(X_s\) is projective because it
is closed in \(P\).

Adjunction gives
\begin{equation}\label{eq:adjunction}
K_{X_s} = (K_P + 2L)|_{X_s} \simeq \cO_{X_s}.
\end{equation}
Moreover, \(X_s\) is a smooth complete intersection of two ample divisors
in \(P=(\P^1)^4\).  The Lefschetz hyperplane theorem therefore gives
\[
\rH^1(X_s,\C) \simeq \rH^1(P,\C)=0.
\]
Hence \(\rH^1(X_s,\cO_{X_s})=0\) by Hodge decomposition. Together with
\eqref{eq:adjunction}, this proves that \(X_s\) is a K3 surface.

\noindent\textbf{Step 2: the intersection form.}
We determine the part of the N\'eron--Severi lattice of \(X_s\) coming
from the ambient product. Pull back the class of \(\cO_{\P^1}(1)\) from each of
the four factors of \(P\). Let \(H_x,H_y,H_z,H_w\) denote the restrictions
of these classes to \(X_s\).
Write \(V_s\coloneqq\Span_{\R}\{H_x,H_y,H_z,H_w\}\) for their real span.
Intersection in \(P\) gives
\[
H_i^2 = 0,
\qquad H_iH_j = 2 \quad \text{for } i \ne j.
\]
The Gram matrix has eigenvalues \(6,-2,-2,-2\), so the four classes are
independent in \(\NS(X_s)\) and span a rank-four lattice of signature
\((1,3)\). In particular, \(\rho(X_s)\geq4\) for every \(s\in\cG\).

\noindent\textbf{Step 3: the involution actions.}
To compute the action on divisor classes, we first recall two push--pull
identities for a finite double cover \(\pi\colon Z\to Y\) of smooth
complex surfaces. Let \(\tau\) be its deck involution.
For divisor classes \(E\) on \(Y\) and \(D\) on \(Z\), one has
\begin{equation}\label{eq:double-cover-push-pull}
\pi_*\pi^*E=2E,
\qquad
\pi^*\pi_*D=D+\tau^*D.
\end{equation}
The first identity is the degree formula for a finite locally free morphism
\cite{StacksProject}*{\href{https://stacks.math.columbia.edu/tag/02RH}
{Tag~02RH}}. For the second, \(\pi\) realizes \(Y\) as the quotient
\(Z/G\), where \(G=\{\id,\tau\}\).
The push--pull formula for a finite group quotient gives
\(\pi^*\pi_*D=\sum_{g\in G}g_*D\)
\cite{Fulton1998}*{Example~1.7.6}.
Since \(\tau^{-1}=\tau\), we have \(\tau_*D=\tau^*D\).
This proves the second identity.

Apply these identities to \(\pi_{yw}\).
Let \(h_y,h_w\) be the pullbacks of the class of \(\cO_{\P^1}(1)\)
from the two factors of \(P_{yw}\coloneqq\P^1_y\times\P^1_w\).
Then \(H_y=\pi_{yw}^*h_y\) and \(H_w=\pi_{yw}^*h_w\), so
\(A_s^*\) fixes \(H_y,H_w\).
To compute its action on \(H_x\), write
\((\pi_{yw})_*H_x=a h_y+b h_w\).
Since \(h_y^2=h_w^2=0\) and \(h_yh_w=1\), the projection formula and
the intersection numbers above give
\[
\begin{aligned}
b&=((\pi_{yw})_*H_x)\mathbin{\cdot}h_y
    =H_x\mathbin{\cdot}\pi_{yw}^*h_y
    =H_x\mathbin{\cdot}H_y=2,\\
a&=((\pi_{yw})_*H_x)\mathbin{\cdot}h_w
    =H_x\mathbin{\cdot}\pi_{yw}^*h_w
    =H_x\mathbin{\cdot}H_w=2.
\end{aligned}
\]
The same computation applies to \(H_z\). Thus
\[
(\pi_{yw})_*H_x = 2h_y + 2h_w,
\qquad
(\pi_{yw})_*H_z = 2h_y + 2h_w.
\]
The second identity in \eqref{eq:double-cover-push-pull} gives
\[
H_x+A_s^*H_x=\pi_{yw}^*(\pi_{yw})_*H_x=2H_y+2H_w,
\qquad
H_z+A_s^*H_z=\pi_{yw}^*(\pi_{yw})_*H_z=2H_y+2H_w.
\]
Consequently,
\begin{equation}\label{eq:A-action}
\begin{aligned}
A_s^*H_x &= -H_x + 2H_y + 2H_w,&
A_s^*H_z &= -H_z + 2H_y + 2H_w,\\
A_s^*H_y &= H_y,& A_s^*H_w &= H_w.
\end{aligned}
\end{equation}
Applying the same argument to \(\pi_{xz}\) gives
\begin{equation}\label{eq:B-action}
\begin{aligned}
B_s^*H_y &= -H_y + 2H_x + 2H_z,&
B_s^*H_w &= -H_w + 2H_x + 2H_z,\\
B_s^*H_x &= H_x,& B_s^*H_z &= H_z.
\end{aligned}
\end{equation}
These formulas show that \(A_s^*\) and \(B_s^*\) preserve \(V_s\).
We use the ordered basis \((H_x,H_y,H_z,H_w)\) and let matrices act
on coefficient column vectors.
The matrices of \(A_s^*|_{V_s}\) and \(B_s^*|_{V_s}\) are
\[
M_A=\begin{pmatrix}
-1&0&0&0\\
2&1&2&0\\
0&0&-1&0\\
2&0&2&1
\end{pmatrix},
\qquad
M_B=\begin{pmatrix}
1&2&0&2\\
0&-1&0&0\\
0&2&1&2\\
0&0&0&-1
\end{pmatrix}.
\]
We also determine the action on the holomorphic two-form.
Since \(A_s\) has order \(2\), we have either
\(A_s^*\Omega_s = \Omega_s\) or \(A_s^*\Omega_s = -\Omega_s\). We show
that \(\Omega_s\) is not invariant under \(A_s\). If a holomorphic
two-form were invariant under \(A_s\), it would descend to
\(\P^1_y \times \P^1_w\).
To check descent across the branch curve, choose local coordinates in
which the cover is \((u,v)\mapsto(u,t=v^2)\).
Invariance of \(f(u,v)\,du\wedge dv\) under \(v\mapsto-v\) says that
\(f\) is odd in \(v\). Thus \(f(u,v)=v g(u,v^2)\) for a holomorphic
function \(g\), and the form is the pullback of
\(\tfrac12 g(u,t)\,du\wedge dt\) on the base \(\P^1 \times \P^1\).
Since
\(\rH^0(\P^1 \times \P^1,K_{\P^1 \times \P^1}) = 0\), \(\Omega_s\)
cannot be \(A_s\)-invariant. Thus \(A_s^*\Omega_s = -\Omega_s\).
The same argument gives \(B_s^*\Omega_s = -\Omega_s\).

\noindent\textbf{Step 4: the action of \(T_s\).}
Since \(T_s^*=A_s^*B_s^*\), its matrix on \(V_s\) in this convention is
\[
M_T=M_AM_B=
\begin{pmatrix}
-1&-2&0&-2\\
2&7&2&8\\
0&-2&-1&-2\\
2&8&2&7
\end{pmatrix}.
\]
Its characteristic polynomial is \((t+1)^2(t^2-14t+1)\).
Thus its spectral radius is \(7+4\sqrt3\).
The orthogonal complement of \(V_s\) in \(\rH^{1,1}(X_s,\R)\) is
negative definite, by the Hodge index theorem, and \(T_s^*\)-invariant.
The action there is an isometry, so its eigenvalues have modulus one.
Thus the first dynamical degree is exactly
\(\lambda_1(T_s)=7+4\sqrt3\), and the Gromov--Yomdin theorem gives
\[
h_{\mathrm{top}}(T_s) = \log\lambda_1(T_s) = \log(7 + 4\sqrt3).
\]

Both involutions \(A_s\) and \(B_s\) pull back \(\Omega_s\) to \(-\Omega_s\), so their composition satisfies
\[
T_s^*\Omega_s = \Omega_s.
\]
\end{proof}

\begin{remark}
For a very general parameter \(s\in\cV\) (that is, outside a countable union
of proper real algebraic subsets), the surface \(X_s\) has Picard number
\(4\) \cite{BhargavaHoKumar2016}*{Theorem~7.1}.
In this case, \(\NS(X_s)\otimes_{\Z}\R=V_s=\Span_{\R}\{H_x,H_y,H_z,H_w\}\).

Bhargava--Ho--Kumar \cite{BhargavaHoKumar2016}*{Section~7} and
Hashimoto--Oda \cite{HashimotoOda2026}*{Section~5} compute the intersection
form and deck-involution actions for very general \(X_s\).
The argument above establishes these formulas for every \(s\in\cG\),
including members of higher Picard number.
\end{remark}

\begin{remark}
\label[remark]{rem:full-cohomology-action}
For every \(s\in\cG\), the characteristic polynomial on the full second
cohomology is
\[
\det\bigl(t\id-T_s^*|_{\rH^2(X_s,\R)}\bigr)
=(t-1)^{18}(t+1)^2(t^2-14t+1).
\]
Indeed, the intersection form is nondegenerate on \(V_s\), so
\(\rH^2(X_s,\R)=V_s\oplus V_s^\perp\). Here \(V_s^\perp\) is the orthogonal
complement of \(V_s\) in \(\rH^2(X_s,\R)\).
Since \(X_s\) is a K3 surface, \(\dim_{\R}V_s^\perp=22-4=18\).
The push--pull identities \eqref{eq:double-cover-push-pull} also hold in
cohomology.
For \(u\in V_s^\perp\), the projection formula shows that
\((\pi_{yw})_*u\) pairs trivially with \(h_y\) and \(h_w\), so it vanishes.
Thus \(A_s^*u=-u\). Similarly, \(B_s^*u=-u\).
Hence \(T_s^*\) is the identity on \(V_s^\perp\).
\end{remark}

\section{A smoothing with prescribed rotation}\label{sec:smoothing}

This section produces good parameters whose real dynamics is a prescribed
Diophantine rotation. We start from reducible surfaces in the family
parametrized by \(\cV\). Their parameters lie outside the good locus
\(\cG\), but their real loci consist of two tori contained in the smooth
locus. On these tori, the real dynamics is an explicit rotation whose
vector depends on a parameter \(\bt\in\R^2\). We perturb the defining
sections in one fixed direction and adjust \(\bt\). This gives the following
theorem.

\begin{theorem}[Prescribed-rotation smoothing]
\label[theorem]{thm:prescribed-rotation-smoothing}
There is an open disk \(Q_0\subset(-\pi,\pi)^2\) with the following
property. For every \(\gamma>0\) and every
\(\balpha\in Q_0\cap\operatorname{DC}(\gamma,3)\), there is a parameter
\(s\in\cG\) such that the smooth real projective K3 surface
\(X_{s,\R}\) has real locus with two connected components
\[
X_{s,\R}(\R)=M_{0,s}\sqcup M_{1,s}.
\]
The components are real-analytic tori, and there are real-analytic
diffeomorphisms \(j_{\ell,s}\colon\T^2\to M_{\ell,s}\) satisfying
\[
T_{s,\R}\circ j_{\ell,s}
=j_{\ell,s}\circ\Rot_{(-1)^\ell\balpha},
\qquad \text{for } \ell=0,1.
\]
\end{theorem}

In \Cref{sec:singular-rotation-models}, we construct the singular
rotation models and describe the dynamics on their real loci.
In \Cref{sec:smoothing-kam}, we choose a smoothing direction whose small
perturbations are good uniformly over a disk of parameters \(\bt\).
The implicit function theorem deforms the tori analytically
(\Cref{lem:graph-persistence}).
A topological degree argument then gives a parameter at which Herman's
correction term vanishes (\Cref{lem:phase-selection}).

\subsection{Dynamics on the singular rotation models}
\label{sec:singular-rotation-models}

We first express rotations in algebraic coordinates.
As in \cite{Moncet2013}*{Section~1.1}, we identify the real projective
line \(\P^1(\R)\) with the one-dimensional torus \(\T\) using Cayley
coordinates:
\[
\kappa\colon\T\xrightarrow{\sim}\P^1(\R),
\qquad
\theta\longmapsto[\cos(\theta/2):\sin(\theta/2)].
\]
In the affine chart, \(\kappa(\theta)=[1:\tan(\theta/2)]\) for
\(\theta\not\equiv\pi\pmod{2\pi}\).
Grouping the \((x,z)\)-coordinates and the \((y,w)\)-coordinates gives
\[
P_\R(\R)\simeq\T^2_{x,z}\times\T^2_{y,w}.
\]
Here \(\T^2_{x,z}\coloneqq\T_x\times\T_z\) and
\(\T^2_{y,w}\coloneqq\T_y\times\T_w\) are copies of \(\T^2\).
Under this identification, \(\pi_{xz}\) and \(\pi_{yw}\) are the
projections onto \(\T^2_{x,z}\) and \(\T^2_{y,w}\), respectively.

For \(t\in\R\), let \(r_t\colon\P^1\to\P^1\) be the projective
automorphism defined over \(\R\) by
\[
r_t([y_0:y_1])=[y_0-ty_1:y_1+ty_0].
\]
Its affine expression is \(r_t(y)=(y+t)/(1-ty)\).
The matrix \(\left(\begin{smallmatrix}1&-t\\t&1\end{smallmatrix}\right)\)
is a positive multiple of a rotation matrix with angle \(\arctan t\).
The point \(\kappa(\theta)\) is the line spanned by
\((\cos(\theta/2),\sin(\theta/2))\), whose angle is \(\theta/2\).
Thus
\begin{equation}\label{eq:cayley-rotation}
r_t\bigl(\kappa(\theta)\bigr)
=\kappa\bigl(\theta+2\arctan t\bigr).
\end{equation}
In particular, \(r_{-t}=r_t^{-1}\).
We use homogeneous coordinates \(x=[x_0:x_1]\), and similarly for
\(y,z,w\).
Define
\[
\begin{aligned}
P_t(x,y)&\coloneqq x_1y_0-x_0y_1-t(x_1y_1+x_0y_0),\\
Q_t(z,w)&\coloneqq z_1w_0-z_0w_1-t(z_1w_1+z_0w_0).
\end{aligned}
\]
These bilinear polynomials vanish precisely when
\(y=r_{-t}(x)\) and \(w=r_{-t}(z)\), respectively.
We use \(r_{-t}\) rather than \(r_t\) so that the dynamics on the graph of
the identity acts by \(r_t\) rather than by its inverse
(\Cref{prop:singular-toral-family}).

We now multiply these polynomials to obtain a reducible surface.
For \(\bt=(t_1,t_2)\in\R^2\), define the corresponding parameter and
rotation vector by
\begin{equation}\label{eq:singular-coefficient-family}
c(\bt)\coloneqq\bigl(P_0Q_{t_2},P_{t_1}Q_0\bigr)\in\cV,
\qquad
\nu(\bt)\coloneqq(2\arctan t_1,2\arctan t_2).
\end{equation}
The map \(c\) is affine in \(\bt\).
The map \(\nu\colon\R^2\to(-\pi,\pi)^2\) is a real-analytic
diffeomorphism. Its Jacobian determinant is
\(4/((1+t_1^2)(1+t_2^2))>0\).
For a map \(g\colon\T^2_{x,z}\to\T^2_{y,w}\), write
\[
\Gamma(g)\coloneqq\{((x,z),g(x,z)):(x,z)\in\T^2_{x,z}\}
\subset\T^2_{x,z}\times\T^2_{y,w}.
\]
We shall see in \Cref{prop:singular-toral-family} that the real locus of
\(X_{c(\bt)}\) is the disjoint union of \(\Gamma(\id)\) and
\(\Gamma(\Rot_{-\nu(\bt)})\).

Before proving this, we compute the dynamics on a real locus consisting of
two graphs.
The computation applies both to the singular models below and to their
smoothings in \Cref{sec:smoothing-kam}.
Let \(s\in\cV\), and suppose that
\[
X_{s,\R}(\R)=\Gamma(g_0)\sqcup\Gamma(g_1).
\]
Here \(g_0,g_1\colon\T^2_{x,z}\to\T^2_{y,w}\) are real-analytic
diffeomorphisms.
Then \(\pi_{yw}\) and \(\pi_{xz}\) restrict to two-sheeted coverings of
\(\T^2_{y,w}\) and \(\T^2_{x,z}\). Each fiber consists of one point on each
graph. The \emph{sheet exchange} of each covering swaps these two points.
The \(yw\)-exchange fixes \((y,w)\), and the \(xz\)-exchange fixes
\((x,z)\).
Their composition is given by
\[
\begin{aligned}
((x,z),g_0(x,z))
&\longmapsto\bigl((g_1^{-1}\circ g_0)(x,z),g_0(x,z)\bigr)\\
&\longmapsto\bigl((g_1^{-1}\circ g_0)(x,z),
(g_0\circ g_1^{-1}\circ g_0)(x,z)\bigr).
\end{aligned}
\]
Thus the \(yw\)-exchange followed by the \(xz\)-exchange acts on the first
graph, in its \((x,z)\)-coordinates, as
\begin{equation}\label{eq:graph-dynamics}
f=g_1^{-1}\circ g_0.
\end{equation}
This is the order of composition defining \(T_s=B_s\circ A_s\).

\begin{proposition}[The singular rotation model]
\label[proposition]{prop:singular-toral-family}
\label[proposition]{lem:central-fiber-geometry}
\label[proposition]{lem:central-etale-double-covers}
\label[proposition]{thm:reducible-graph-model}
Let \(\bt=(t_1,t_2)\in\R^2\) satisfy \(t_1t_2\ne0\).
The complex surface \(X_{c(\bt)}\) is singular.
Its real locus consists of two connected components
\begin{equation}\label{eq:central-real-graphs}
X_{c(\bt),\R}(\R)
=M_0\sqcup M_1
=\Gamma(\id)\sqcup\Gamma(\Rot_{-\nu(\bt)})
\subset X_{c(\bt)}^{\mathrm{sm}}.
\end{equation}
The graph parametrizations depend real-analytically on \(\bt\).
Both projections \(\pi_{xz}\) and \(\pi_{yw}\) restrict to real-algebraic
isomorphisms on each component.
They are therefore split two-sheeted real-analytic coverings of their real
bases. Their sheet exchanges are real-analytic involutions.
The \(yw\)-exchange followed by the \(xz\)-exchange acts in the
\((x,z)\)-coordinates of \(M_0\) and \(M_1\), respectively, as
\[
\Rot_{\nu(\bt)}
\quad\text{and}\quad
\Rot_{-\nu(\bt)}.
\]
\end{proposition}

\begin{proof}
We first determine the real locus.
The defining equations \(P_0Q_{t_2}\) and \(P_{t_1}Q_0\) of
\(X_{c(\bt)}\) factor.
Pairing one factor from each gives the set-theoretic decomposition in
\(P=(\P^1)^4\):
\begin{equation}\label{eq:central-decomposition}
X_{c(\bt)}
=V(P_0,Q_0)\cup V(P_{t_1},Q_{t_2})
 \cup V(P_0,P_{t_1})\cup V(Q_0,Q_{t_2}).
\end{equation}
For \(t\ne0\), the fixed points of \(r_t\) are
\(\Sigma\coloneqq\{[1:i],[1:-i]\}\).
Thus the last two sets lie over \(x=y\in\Sigma\) and
\(z=w\in\Sigma\), respectively. They have no real points.
The first two sets give the graphs in \eqref{eq:central-real-graphs}.
These graphs are disjoint because \(r_{t_1}\) has no real fixed point.
Their formulas give real-analytic dependence on \(\bt\).
All four sets in \eqref{eq:central-decomposition} have complex dimension
two. At a point \((q,q,r,r)\) with \(q,r\in\Sigma\), all four factors
vanish. The differentials of both defining equations vanish there, so the
surface is singular.

To deform \(M_0\) and \(M_1\) analytically, we will use the Jacobian in the
\((y,w)\)-variables. We verify that it is invertible.
Along \(M_0\), differentiation in affine charts gives
\[
d(P_0Q_{t_2})=Q_{t_2}\,dP_0,
\qquad
d(P_{t_1}Q_0)=P_{t_1}\,dQ_0.
\]
The coefficients \(Q_{t_2}\) and \(P_{t_1}\) are nowhere zero there.
Since \(P_0\) and \(Q_0\) are linear in \(y\) and \(w\), respectively,
their derivatives in these variables are nonzero along \(M_0\).
Thus the \((y,w)\)-Jacobian is diagonal with nonzero entries.
Along \(M_1\), the corresponding formulas are
\[
d(P_0Q_{t_2})=P_0\,dQ_{t_2},
\qquad
d(P_{t_1}Q_0)=Q_0\,dP_{t_1}.
\]
Here \(P_0\) and \(Q_0\) are nowhere zero, so the \((y,w)\)-Jacobian is
antidiagonal with nonzero entries.
In particular, it is invertible along \(M_0\) and \(M_1\).
This also proves that they lie in the smooth locus.

The two algebraic graphs are given by
\((y,w)=(x,z)\) and
\((y,w)=(r_{-t_1}(x),r_{-t_2}(z))\).
Both coordinate projections restrict to real-algebraic isomorphisms on
these graphs. Their restrictions to the real locus are consequently split
two-sheeted coverings, with real-analytic sheet exchanges.
Finally, \eqref{eq:graph-dynamics} with
\(g_0=\id\) and \(g_1=\Rot_{-\nu(\bt)}\) gives the first rotation.
Interchanging the graphs gives the opposite rotation on \(M_1\).
\end{proof}

\subsection{The smoothing pencil and elimination of the correction term}
\label{sec:smoothing-kam}

We now perturb the defining sections in one fixed direction.
The parameter \(\bt\in\R^2\) adjusts the rotation, while a real
parameter \(\lambda\) controls the smoothing.
For \(v_0\in\cV\setminus\{0\}\), write
\[
s(\bt,\lambda)\coloneqq c(\bt)+\lambda v_0.
\]
For fixed \(\bt\), we call \(X_{c(\bt)}\) the \emph{central fiber}.
The next lemma gives \(v_0\) and a disk of parameters on which every
sufficiently small nonzero \(\lambda\) gives a parameter in \(\cG\).
We call \(v_0\) a \emph{smoothing direction} on this disk.
The bound on \(\lambda\) must be uniform on the disk. Indeed, we will fix
\(\lambda\) first and then adjust \(\bt\) to eliminate Herman's correction
term. The adjusted parameter must still lie in \(\cG\).

\begin{lemma}[A uniformly good smoothing pencil]
\label[lemma]{lem:second-smoothing-pencil}
Let \(B\Subset\{\bt\in\R^2:t_1t_2\ne0\}\) be a nonempty open rectangle.
There are a vector \(v_0\in\cV\setminus\{0\}\), an open disk
\(D_0\Subset B\), and a constant \(\delta_0>0\) such that
\begin{equation}\label{eq:second-good-pencil}
s(\bt,\lambda)\in\cG
\qquad
\text{for } \bt\in\overline{D_0} \text{ and } 0<\abs\lambda<\delta_0.
\end{equation}
\end{lemma}

\begin{proof}
We detect good parameters by a polynomial and choose a direction on which
it does not vanish identically.
By \Cref{prop:good-locus}, the complement of \(\cG\) is a proper real
algebraic set. Choose real polynomials \(p_1,\ldots,p_r\) defining it.
The polynomial \(\Delta\coloneqq p_1^2+\cdots+p_r^2\) then satisfies
\(\cG=\{\Delta\ne0\}\) on \(\cV\).
Choose \(\bt^*\in B\) and \(s^*\in\cG\), and set
\[
v_0\coloneqq s^*-c(\bt^*).
\]
The central fibers are singular by \Cref{prop:singular-toral-family},
so \(v_0\ne0\).

Since \(c\) is affine, \(\Delta(c(\bt)+\lambda v_0)\) is polynomial
in \((\bt,\lambda)\). It vanishes for \(\lambda=0\), but its value at
\((\bt^*,1)\) is \(\Delta(s^*)\ne0\).
Let \(k\geq1\) be the smallest integer for which the coefficient of
\(\lambda^k\) is not identically zero on \(B\).
Factoring out \(\lambda^k\), we obtain
\begin{equation}\label{eq:second-pencil-leading-term}
\Delta(c(\bt)+\lambda v_0)
=\lambda^k\bigl(a(\bt)+\lambda R(\bt,\lambda)\bigr).
\end{equation}
Here \(a\) and \(R\) are polynomials, and \(a\) is not identically zero.
Choose an open disk \(D_0\) with
\(\overline{D_0}\subset B\cap\{a\ne0\}\).
The function \(\abs a\) has a positive minimum on \(\overline{D_0}\),
and \(R\) is bounded near \(\overline{D_0}\times\{0\}\).
Thus there is \(\delta_0>0\) such that
\[
\abs{\lambda R(\bt,\lambda)}<\tfrac12\abs{a(\bt)}
\qquad
\text{for } \bt\in\overline{D_0} \text{ and } \abs\lambda<\delta_0.
\]
The factor in parentheses in \eqref{eq:second-pencil-leading-term} is
nonzero throughout this set.
This proves \eqref{eq:second-good-pencil}, with \(\delta_0\) independent
of \(\bt\).
\end{proof}

Next, we deform the two tori in the real locus of a singular rotation
model along a pencil. The deformation works for any direction \(v_0\). The resulting
dynamics agrees with \(T_{s,\R}\) only at good parameters \(s\).

\begin{lemma}[Persistence of the two graphs]
\label[lemma]{lem:graph-persistence}
Let \(v_0\in\cV\), and let \(\bt_0=(t_{0,1},t_{0,2})\in\R^2\) satisfy
\(t_{0,1}t_{0,2}\ne0\). There are an open disk \(D\) centered at
\(\bt_0\) and a constant \(\delta>0\) with the following properties for
all \((\bt,\lambda)\in\overline D\times[-\delta,\delta]\).
\begin{enumerate}[label=\textup{(\roman*)}]
\item\textup{\textbf{Real locus.}}
There are real-analytic diffeomorphisms
\(g_0(\bt,\lambda),g_1(\bt,\lambda)\colon\T^2_{x,z}\to\T^2_{y,w}\) such
that
\begin{equation}\label{eq:persistent-real-tori}
X_{s(\bt,\lambda),\R}(\R)
=\Gamma(g_0(\bt,\lambda))\sqcup\Gamma(g_1(\bt,\lambda)).
\end{equation}
Both graphs lie in the smooth locus of \(X_{s(\bt,\lambda)}\).
The maps \(g_\ell\) and their inverses depend real-analytically on
\((\bt,\lambda)\) and the torus variable. Moreover,
\(g_0(\bt,0)=\id\) and \(g_1(\bt,0)=\Rot_{-\nu(\bt)}\).
\item\textup{\textbf{Dynamics.}}
The map
\begin{equation}\label{eq:persistent-central-rotation}
f_{\bt,\lambda}\coloneqq
g_1(\bt,\lambda)^{-1}\circ g_0(\bt,\lambda)
\end{equation}
belongs to \(\operatorname{Diff}^{\omega}_0(\T^2)\) and depends
real-analytically on \((\bt,\lambda)\) and the torus variable.
It satisfies \(f_{\bt,0}=\Rot_{\nu(\bt)}\).
If \(s(\bt,\lambda)\in\cG\), then \(f_{\bt,\lambda}\) represents the
restriction of \(T_{s(\bt,\lambda),\R}\) to
\(\Gamma(g_0(\bt,\lambda))\) in its \((x,z)\)-coordinates.
\end{enumerate}
\end{lemma}

\begin{proof}
At \((\bt_0,0)\), the real locus is
\(\Gamma(\id)\sqcup\Gamma(\Rot_{-\nu(\bt_0)})\).
Write \(s=s(\bt,\lambda)\). In affine charts of \(P_\R\), the surface
\(X_{s,\R}\) is given by the two polynomial equations \(F_{1,s}=F_{2,s}=0\).
At \((\bt_0,0)\), their Jacobian in the \((y,w)\)-variables is invertible
along both graphs by \Cref{prop:singular-toral-family}.
We apply the real-analytic implicit function theorem to these equations.
Near each point of the two graphs, it solves the equations for \((y,w)\)
in terms of \((\bt,\lambda,x,z)\).
By compactness, these local solutions exist on a common parameter neighborhood of
\((\bt_0,0)\). Choose disjoint open neighborhoods \(N_0\) and \(N_1\) of
\(\Gamma(\id)\) and \(\Gamma(\Rot_{-\nu(\bt_0)})\), respectively, on which
these solutions are unique.
The local solutions therefore agree on overlaps and glue to maps
\(g_\ell(\bt,\lambda)\colon\T^2_{x,z}\to\T^2_{y,w}\), for \(\ell=0,1\).
These maps depend real-analytically on the parameters and the torus
variable.
After shrinking the parameter neighborhood, the real locus has no points
outside \(N_0\cup N_1\).
Otherwise, compactness of \(P_\R(\R)\) would give such points converging,
as \((\bt,\lambda)\to(\bt_0,0)\), to a real point of \(X_{c(\bt_0)}\)
outside \(N_0\cup N_1\), which is impossible.
Choose a small disk \(D\) centered at \(\bt_0\) and \(\delta>0\) so that
this parameter neighborhood contains all \((\bt,\lambda)\in\overline
D\times[-\delta,\delta]\).
This gives \eqref{eq:persistent-real-tori}.
The \((y,w)\)-Jacobians remain invertible, so both graphs lie in the
smooth locus.

For \(\bt\in\overline D\), \Cref{prop:singular-toral-family} gives
\(X_{c(\bt),\R}(\R)=\Gamma(\id)\sqcup\Gamma(\Rot_{-\nu(\bt)})\).
The graph \(\Gamma(\id)\) lies in \(N_0\), so it equals
\(\Gamma(g_0(\bt,0))\). Hence \(g_0(\bt,0)=\id\) and
\(g_1(\bt,0)=\Rot_{-\nu(\bt)}\).
After shrinking \(D\) and \(\delta\), the graph maps are \(C^1\)-close
to these rotations and hence are diffeomorphisms.
The parameterized inverse function theorem gives real-analytic dependence
of their inverses as well. This proves~\textup{(i)}.

The map \(f_{\bt,\lambda}\) therefore depends real-analytically on all
variables, and \(f_{\bt,0}=\Rot_{\nu(\bt)}\).
Varying the smoothing parameter from \(0\) to \(\lambda\) joins this
rotation to \(f_{\bt,\lambda}\).
Since rotations are isotopic to the identity,
\(f_{\bt,\lambda}\in\operatorname{Diff}^{\omega}_0(\T^2)\).
Now suppose that \(s=s(\bt,\lambda)\in\cG\).
By~\textup{(i)}, the projections \(\pi_{yw}\) and \(\pi_{xz}\) from the
real locus have exactly two preimages over every point of
\(\T^2_{y,w}\) and \(\T^2_{x,z}\), respectively, one on each graph.
The deck involutions \(A_{s,\R}\) and \(B_{s,\R}\) exchange these
points. Thus \eqref{eq:graph-dynamics} shows that \(f_{\bt,\lambda}\)
represents the restriction of \(T_{s,\R}\) to the first component in
its \((x,z)\)-coordinates.
\end{proof}

The last step adjusts \(\bt\) so that Herman's correction term vanishes
after smoothing. At \(\lambda=0\), the correction term has a nondegenerate
zero. The correction term is only known to depend continuously on the
parameters. The next lemma therefore uses the topological degree to show
that this zero persists.

\begin{lemma}[Persistence of a zero]
\label[lemma]{lem:phase-selection}
Let \(D\subset\R^2\) be an open disk, let \(\delta>0\), and let
\(\Theta\colon\overline D\times[-\delta,\delta]\to\R^2\) be continuous.
Assume that \(\Theta(\,\cdot\,,0)\) is \(C^1\) on \(D\) and has exactly one
zero \(\bt_0\) in \(\overline D\), which lies in \(D\).
Assume also that its Jacobian determinant at \(\bt_0\) is nonzero.
Then there is \(0<\delta_1\leq\delta\) with the following property.
For every \(\lambda\) with \(\abs\lambda<\delta_1\), there exists
\(\bt_\lambda\in D\) such that \(\Theta(\bt_\lambda,\lambda)=\bzero\).
\end{lemma}

\begin{proof}
We show that \(\Theta(\,\cdot\,,0)\) has nonzero degree on \(D\) and that
small perturbations keep this degree.
Since \(\bt_0\) is the only zero and it is nondegenerate, \(\bzero\) is a
regular value of \(\Theta(\,\cdot\,,0)\) on \(D\).
By the regular-value formula, \(\deg(\Theta(\,\cdot\,,0),D,\bzero)\) is
the sign of the Jacobian determinant at \(\bt_0\), hence \(\pm1\).
The boundary values at \(\lambda=0\) are separated from zero by
\[
m\coloneqq\min_{\bt\in\partial D}\abs{\Theta(\bt,0)}>0.
\]
Uniform continuity on the compact set \(\overline D\times[-\delta,\delta]\)
gives \(0<\delta_1\leq\delta\) such that
\[
\abs{\Theta(\bt,\lambda)-\Theta(\bt,0)}<m
\qquad\text{for } \bt\in\partial D \text{ and } \abs\lambda<\delta_1.
\]
Fix such a \(\lambda\).
The straight-line homotopy from \(\Theta(\,\cdot\,,0)\) to
\(\Theta(\,\cdot\,,\lambda)\) therefore avoids zero on \(\partial D\).
By homotopy invariance of the degree, we have
\[
\deg(\Theta(\,\cdot\,,\lambda),D,\bzero)
=\deg(\Theta(\,\cdot\,,0),D,\bzero)\ne0.
\]
Hence \(\Theta(\,\cdot\,,\lambda)\) has a zero \(\bt_\lambda\in D\).
\end{proof}

We now combine
\Cref{lem:second-smoothing-pencil,lem:graph-persistence,lem:phase-selection}
with Herman's theorem.
Recall the real-analytic diffeomorphism
\(\nu\colon\R^2\to(-\pi,\pi)^2\) from
\eqref{eq:singular-coefficient-family}, given by
\(\nu(\bt)=(2\arctan t_1,2\arctan t_2)\).
Its inverse is
\(\nu^{-1}(\bbeta)=(\tan(\beta_1/2),\tan(\beta_2/2))\).

\begin{proof}[Proof of \Cref{thm:prescribed-rotation-smoothing}]
We will find the required parameter in the form \(s(\bt,\lambda)\).
Its data are chosen in the following order:
\[
\underset{\makebox[0pt]{\scriptsize\Cref{lem:second-smoothing-pencil}}}{(v_0,D_0,\delta_0)}
\quad\longrightarrow\quad(\balpha, \bt_0=\nu^{-1}(\balpha))
\quad\longrightarrow\quad
\underset{\makebox[0pt]{\scriptsize\Cref{lem:graph-persistence}}}{(D,\delta)}
\quad\longrightarrow\quad\lambda
\quad\longrightarrow\quad
\underset{\makebox[0pt]{\scriptsize\Cref{lem:phase-selection}}}{\bt_\lambda}.
\]
Here \Cref{lem:graph-persistence} gives the disk \(D\) centered at
\(\bt_0\) and the constant \(\delta>0\).
We shrink \(D\) and \(\delta\) so that the smoothed parameters stay in
\(\cG\), Herman's theorem applies, and the zero of the correction term
persists.
We then fix \(0<\abs\lambda<\delta\) and choose \(\bt_\lambda\) to cancel
the correction term.

Fix a nonempty open rectangle \(B\Subset\{t_1t_2\ne0\}\), and choose
\(v_0,D_0,\delta_0\) as in \Cref{lem:second-smoothing-pencil}.
Since \(\nu\) is a diffeomorphism, we can choose an open disk
\(Q_0\Subset\nu(D_0)\).
Fix \(\gamma>0\) and a prescribed vector
\(\balpha\in Q_0\cap\operatorname{DC}(\gamma,3)\).
Set \(\bt_0\coloneqq\nu^{-1}(\balpha)\).
For every sufficiently small nonzero \(\lambda\), we will choose
\(\bt_\lambda\) near \(\bt_0\) and set
\[
s_\lambda\coloneqq s(\bt_\lambda,\lambda)
=c(\bt_\lambda)+\lambda v_0.
\]
By \Cref{lem:second-smoothing-pencil}, this parameter belongs to \(\cG\)
whenever \(\bt_\lambda\in\overline{D_0}\) and \(0<\abs\lambda<\delta_0\).

Apply \Cref{lem:graph-persistence} to \(v_0\) and \(\bt_0\).
After shrinking the resulting disk \(D\) and constant \(\delta\), we may
assume that \(\overline D\subset D_0\) and \(\delta<\delta_0\).
Then \(s(\bt,\lambda)\in\cG\) for all \(\bt\in\overline D\) and
\(0<\abs\lambda<\delta\). For these parameters,
\(f_{\bt,\lambda}\) represents the restriction of
\(T_{s(\bt,\lambda),\R}\) to the first real component.
It remains to adjust \(\bt\) so that Herman's correction term vanishes.

Apply \Cref{thm:herman} to \(\balpha\), with \(\eta=\pi/2\).
It gives an integer \(q\), a \(C^q\)-open neighborhood \(\cU_{\balpha}\),
and a continuous correction term \(\vartheta_{\balpha}\).
Shrink \(D\) so that
\(\norm{\nu(\bt)-\balpha}_\infty<r_{\balpha}\) on \(\overline D\).
Then \(f_{\bt,0}\in\cU_{\balpha}\) throughout \(\overline D\).
Since \(f_{\bt,\lambda}\) depends real-analytically on the parameters and
the torus variable, and \(\T^2\) is compact, it depends continuously on
\((\bt,\lambda)\) in the \(C^q\)-topology.
By compactness of \(\overline D\), decrease \(\delta\) so that
\(f_{\bt,\lambda}\in\cU_{\balpha}\) for all
\((\bt,\lambda)\in\overline D\times[-\delta,\delta]\).
We can therefore define the correction term map
\[
\Theta(\bt,\lambda)\coloneqq
\vartheta_{\balpha}(f_{\bt,\lambda}).
\]
This map is continuous, and \eqref{eq:herman-rotation-counterterm} gives
\[
\Theta(\bt,0)=\nu(\bt)-\balpha.
\]

At \(\lambda=0\), this map is real-analytic in \(\bt\).
Since \(\nu\) is a diffeomorphism, \(\bt_0\) is its only zero in
\(\overline D\), and its Jacobian determinant there is nonzero.
Apply \Cref{lem:phase-selection}, and decrease \(\delta\) accordingly.
For each \(0<\abs\lambda<\delta\), the lemma gives \(\bt_\lambda\in D\)
with \(\Theta(\bt_\lambda,\lambda)=\bzero\).
This chooses the parameter \(s_\lambda\in\cG\) introduced at the start.

Since the correction term vanishes, Herman's theorem gives
\(\psi_\lambda\in\operatorname{Diff}^{\omega}_0(\T^2)\) such that
\begin{equation}\label{eq:second-exact-conjugacy}
f_{\bt_\lambda,\lambda}
=\psi_\lambda\circ\Rot_{\balpha}\circ\psi_\lambda^{-1}.
\end{equation}
Fix such a \(\lambda\), set \(s=s_\lambda\), and denote the two graphs
in \eqref{eq:persistent-real-tori} by \(M_{0,s}\) and \(M_{1,s}\).
In the fixed Cayley coordinates, the conjugacy on the first component is
\[
j_{0,s}(x,z)\coloneqq
\bigl(\psi_\lambda(x,z),
g_0(\bt_\lambda,\lambda)(\psi_\lambda(x,z))\bigr).
\]
By \Cref{lem:graph-persistence}\textup{(ii)}, this map satisfies
\begin{equation}\label{eq:persistent-good-compatibility}
T_{s,\R}\circ j_{0,s}=j_{0,s}\circ\Rot_{\balpha}.
\end{equation}
The involution \(A_{s,\R}\) exchanges the two components and satisfies
\(A_{s,\R}T_{s,\R}A_{s,\R}=T_{s,\R}^{-1}\).
Hence \(j_{1,s}\coloneqq A_{s,\R}\circ j_{0,s}\) satisfies
\[
T_{s,\R}\circ j_{1,s}=j_{1,s}\circ\Rot_{-\balpha}.
\]
Together with \(s\in\cG\), this proves the theorem.
\end{proof}

\section{Proofs of the main results}\label{sec:proof-main}

We now assemble the ingredients. The smoothing theorem supplies the real
dynamics, \Cref{prop:k3-entropy} supplies the entropy, and the
complexification of the rotation tori supplies the invariant biannuli.
At the end of the section, we prove that the two tori of the real locus
lie in distinct Fatou components.

\begin{proof}[Proof of \Cref{thm:complex-main} and \Cref{thm:real-main}]
Let \(Q_0\) be the disk in \Cref{thm:prescribed-rotation-smoothing}.
By \Cref{lem:diophantine-choice}, choose \(\gamma>0\) and
\(\balpha\in Q_0\cap\operatorname{DC}(\gamma,3)\).
Let \(s\in\cG\) be the corresponding parameter given by
\Cref{thm:prescribed-rotation-smoothing}, and set
\[
X_{\R}\coloneqq X_{s,\R},
\qquad T_{\R}\coloneqq T_{s,\R}.
\]
Then \(X_{\R}\) is a smooth projective real K3 surface, and
\Cref{prop:k3-entropy} gives
\[
\rho(X)\geq4,
\qquad
h_{\mathrm{top}}(T)=\log(7+4\sqrt3)>0.
\]

Set \(M_\ell\coloneqq M_{\ell,s}\). The smoothing theorem gives
\[
X_{\R}(\R)=M_0\sqcup M_1.
\]
It also gives real-analytic conjugacies
\(j_{\ell,s}\colon\T^2\to M_\ell\).
Using the identification \(\T^2\simeq(\S^1)^2\) from the introduction,
regard \(j_{\ell,s}\) as a map from \((\S^1)^2\) and denote it by
\(j_{\ell,\R}\). The rotation conjugacies become
\[
T_{\R}\circ j_{\ell,\R}
=j_{\ell,\R}\circ R_{(-1)^\ell\balpha},
\qquad \text{for } \ell=0,1.
\]
Thus both components are invariant, and the real dynamics has zero
topological entropy because each rotation does.

As components of the real locus of the smooth real surface \(X_{\R}\),
the tori \(M_\ell\) are maximally totally real in \(X\). Apply
\Cref{prop:herman-ring-complexify-torus} to \(T\) and \(j_{\ell,\R}\).
The tori \(M_0\) and \(M_1\) are disjoint and compact. We may therefore
shrink the two annular neighborhoods to a common width \(r>0\) with
disjoint images. The resulting holomorphic embeddings satisfy
\[
j_\ell\colon\cA_r\longrightarrow X,
\qquad j_\ell|_{(\S^1)^2}=j_{\ell,\R},
\qquad \text{for } \ell=0,1.
\]
They conjugate \(T\) to the rotations:
\[
T\circ j_\ell=j_\ell\circ R_{(-1)^\ell\balpha}.
\]
Set \(W_\ell\coloneqq j_\ell(\cA_r)\). These disjoint open sets satisfy
\[
X_{\R}(\R)\subset W_0\cup W_1\subset\Fat(T).
\]
This proves \Cref{thm:real-main}. Forgetting the real structure and taking
\(W_0\) proves \Cref{thm:complex-main}.
\end{proof}

It remains to prove \Cref{prop:fatou-components-and-coupling}.
The key step extends the rotations on a biannulus to the whole Fatou
component containing it.
This uses results on rotation domains, which we now recall.

\begin{definition}[Rotation domain]\label[definition]{dfn:rotation-domain}
Let \(f\) be an automorphism of a compact complex surface \(Y\).
An \(f\)-invariant Fatou component \(U\) is a \emph{rotation domain}
if there are integers \(n_j\) with \(\abs{n_j}\to\infty\) such that
\[
f^{n_j}|_U\longrightarrow\id_U
\quad\text{uniformly on compact subsets of }U.
\]
\end{definition}

The next theorem follows from Propositions~1.1 and~1.3 and Theorem~1.4
of Bedford--Kim \cite{BedfordKim2012}.

\begin{theorem}[Volume-preserving automorphisms]
\label[theorem]{thm:rotation-domain-structure}
Let \(f\) be an automorphism of a compact complex surface \(Y\) that
preserves a smooth positive volume form.
\begin{enumerate}[label=\textup{(\roman*)}]
\item\textup{\textbf{Rotation domains.}}
Every \(f\)-invariant Fatou component \(U\) is a rotation domain.
\item\textup{\textbf{Group action.}}
For each such \(U\), the closure of \(\{f^n|_U:n\in\Z\}\) in
\(\operatorname{Aut}(U)\), taken in the compact-open topology, is a
compact abelian Lie group.
Its action on \(U\) is real-analytic.
\end{enumerate}
\end{theorem}

We now outline the proof of \Cref{prop:fatou-components-and-coupling}.
Neither the disjointness of \(W_0\) and \(W_1\) nor the opposite rotation
vectors separate their Fatou components.
The additional ingredient is the sign \(A^*\Omega=-\Omega\) from
\Cref{prop:k3-entropy}.
Suppose that both tori lay in one Fatou component \(U\).
By \Cref{thm:rotation-domain-structure}, the rotations on \(W_0\) would
then extend to an action of \((\S^1)^2\) on all of \(U\).
Evaluating \(\Omega\) on the two infinitesimal generators of this action
would give a nonzero constant on \(U\).
The involution \(A\) would preserve \(U\) and reverse both generators.
These two sign changes cancel, so they leave the constant unchanged.
Since \(A\) also negates \(\Omega\), the constant would equal its own
negative, a contradiction.

\begin{proof}[Proof of \Cref{prop:fatou-components-and-coupling}]
The construction in \Cref{thm:prescribed-rotation-smoothing} gives the
real deck involution \(A\coloneqq A_s\) with
\[
A(M_0)=M_1,\qquad ATA=T^{-1}.
\]
Thus \(AT^nA=T^{-n}\) for every \(n\in\Z\), so \(A\) maps \(\Fat(T)\) onto
itself. Hence \(A\) permutes the Fatou components, and \(A(U_0)=U_1\).
Restricting \(AT=T^{-1}A\) to \(M_0\) gives the stated conjugacy.

Suppose, for contradiction, that \(U_0=U_1=:U\).
The component \(U\) is \(T\)-invariant because it contains the invariant
biannulus \(W_0\). The automorphism \(T\) preserves the canonical volume
\(\vol_X\) in \eqref{eq:canonical-volume}, which is given by a smooth
positive volume form.
Thus \Cref{thm:rotation-domain-structure}\textup{(i)} shows that \(U\) is
a rotation domain.
By \Cref{thm:rotation-domain-structure}\textup{(ii)}, the closure
\[
G\coloneqq\overline{\{T^n|_U:n\in\Z\}}
\subset\operatorname{Aut}(U)
\]
is a compact abelian Lie group acting real-analytically on \(U\).

We first identify \(G\) with the rotation group of \(\cA_r\), which is \((\S^1)^2\).
Each \(g\in G\) is a locally uniform limit of iterates \(T^{n_j}|_U\).
On \(\cA_r\), we have \(T^{n_j}\circ j_0=j_0\circ R_{n_j\balpha}\).
Passing to a subsequence with \(R_{n_j\balpha}\to R_{\bbeta}\) gives
\(g\circ j_0=j_0\circ R_{\bbeta}\).
Hence \(g\mapsto j_0^{-1}\circ g|_{W_0}\circ j_0\) is a continuous
homomorphism from \(G\) to the rotation group \((\S^1)^2\) of \(\cA_r\).
Its image is compact and contains the dense subgroup
\(\{R_{n\balpha}:n\in\Z\}\), so it is surjective.
The identity theorem gives injectivity.
Since \(G\) is compact, this map is an isomorphism of topological groups.
For \(\bbeta\in\R^2\), let \(g_{\bbeta}\in G\) be the element with
\(g_{\bbeta}\circ j_0=j_0\circ R_{\bbeta}\) on \(\cA_r\).
Let \(\xi_1\) and \(\xi_2\) be the vector fields on \(U\) generating the
flows \(t\mapsto g_{(t,0)}\) and \(t\mapsto g_{(0,t)}\), respectively.
They are holomorphic because each element of \(G\) is holomorphic.

Choose a nonzero holomorphic two-form \(\Omega\) on \(X\).
By \Cref{prop:k3-entropy}, we have \(T^*\Omega=\Omega\) and
\(A^*\Omega=-\Omega\).
Passing to the limit in \((T^{n_j})^*\Omega=\Omega\) shows that \(G\)
preserves \(\Omega\).
Since \(G\) is abelian, the generators satisfy
\[
[\xi_1,\xi_2]=0,
\qquad
\mathcal L_{\xi_1}\Omega=\mathcal L_{\xi_2}\Omega=0.
\]
Define the holomorphic function \(h\colon U\to\C\) by
\(h\coloneqq\Omega(\xi_1,\xi_2)\).
Since \(\Omega\) is closed, Cartan's formula
\(\mathcal L_{\xi_i}=d\iota_{\xi_i}+\iota_{\xi_i}d\) gives
\(d(\iota_{\xi_i}\Omega)=\mathcal L_{\xi_i}\Omega=0\) for \(i=1,2\).
Combining this with the identity
\(\iota_{[\xi_1,\xi_2]}=\mathcal L_{\xi_1}\iota_{\xi_2}
-\iota_{\xi_2}\mathcal L_{\xi_1}\), we obtain
\[
0=\iota_{[\xi_1,\xi_2]}\Omega
=\mathcal L_{\xi_1}(\iota_{\xi_2}\Omega)
=d(\iota_{\xi_1}\iota_{\xi_2}\Omega)
+\iota_{\xi_1}d(\iota_{\xi_2}\Omega)
=-dh.
\]
Since \(U\) is connected, \(h\) is constant.
The flows \(t\mapsto R_{(t,0)}\) and \(t\mapsto R_{(0,t)}\) are generated
by \(iz_1\partial_{z_1}\) and \(iz_2\partial_{z_2}\).
Hence \(\xi_1\) and \(\xi_2\) are their pushforwards under \(j_0\) on
\(W_0\).
They are therefore linearly independent over \(\C\) at every point of
\(W_0\).
The form \(\Omega\) is nowhere vanishing, so this constant is nonzero.

Since \(A(U)=U\), the relation \(AT^nA=T^{-n}\) extends by continuity to
\(AgA=g^{-1}\) for every \(g\in G\).
In particular, \(Ag_{\bbeta}A=g_{-\bbeta}\) for every \(\bbeta\in\R^2\).
Differentiating the two flows at \(t=0\) gives
\[
dA_x\bigl(\xi_i(x)\bigr)=-\xi_i(Ax),
\qquad \text{for } x\in U \text{ and } i=1,2.
\]
Since \(\Omega\) is bilinear, the two minus signs cancel.
For every \(x\in U\), we obtain
\[
\begin{aligned}
h(Ax)
&=\Omega_{Ax}\bigl(\xi_1(Ax),\xi_2(Ax)\bigr)\\
&=(A^*\Omega)_x\bigl(\xi_1(x),\xi_2(x)\bigr)
=-h(x).
\end{aligned}
\]
This contradicts the fact that \(h\) is a nonzero constant.
Therefore \(U_0\ne U_1\).
\end{proof}

\begin{remark}
For every nonzero holomorphic two-form \(\Omega\) defined over \(\R\),
the two real components have equal area:
\[
\int_{M_0}|\Omega|=\int_{M_1}|\Omega|.
\]
Here \(|\Omega|\) is the area measure defined by the restriction of
\(\Omega\) to each real component. The equality follows from
\(A(M_0)=M_1\) and \(A^*\Omega=-\Omega\).
\end{remark}

\section{A conjecture on metric entropy}
\label{sec:metric-entropy-conjecture}

Let \(T\) be an automorphism of a K3 surface \(X\). Recall that the
canonical volume \(\vol_X\) in \eqref{eq:canonical-volume} is a smooth
\(T\)-invariant probability measure.
In addition to topological entropy, one can therefore consider the metric
entropy of \(T\) with respect to \(\vol_X\). Fix a Riemannian metric \(g\) on \(X\).
By Pesin's entropy formula
\cite{Pesin1977},
\[
h_{\vol_X}(T)
= 2\int_X \lambda^+(x)\,\di\vol_X(x),
\qquad
\lambda^+(x)
\coloneqq \lim_{N \to \infty}\frac{1}{N}\log\norm{D_xT^N}_g.
\]
The limit defining
\(\lambda^+(x)\) exists for \(\vol_X\)-almost every \(x\) and is independent
of the choice of \(g\). The factor \(2\) reflects the complex multiplicity
of the positive Lyapunov exponent.

Cantat asked in his thesis whether positive topological entropy implies
positive entropy with respect to \(\vol_X\)
\cite{CantatThesis1999}*{Chapter~3}.
Tosatti later stated this as a conjecture
\cite{Tosatti2021}*{Conjecture~7.1}.

\begin{conjecture}\label[conjecture]{conj:positive-canonical-entropy}
Let \(X\) be a projective K3 surface and let \(T\) be
an automorphism of \(X\) with positive topological entropy. Then
\[
h_{\vol_X}(T)>0.
\]
\end{conjecture}

Since \((X,T,\vol_X)\) is a volume-preserving dynamical system, the previous
conjecture could be regarded as a holomorphic analogue of the longstanding
positive metric entropy conjecture for the Chirikov standard map. Although \Cref{conj:positive-canonical-entropy}
appears difficult, we can raise the following question about the
metric entropy of the examples constructed in 
\Cref{thm:real-main}.

\begin{question}\label[question]{ques:special-case-positive-canonical-entropy}
Let \((X,T)\) be a pair obtained in \Cref{thm:real-main}. Is
\[
h_{\vol_X}(T)>0?
\]
\end{question}

For every such pair, \Cref{cor:canonical-volume-nonergodic} shows that \(\vol_X\) is not ergodic. A positive
answer to \Cref{ques:special-case-positive-canonical-entropy} would therefore
give a negative answer to \cite{Tosatti2021}*{Question~7.2}, which asks whether
positivity of the canonical-volume entropy implies ergodicity.

\appendix
\section{Penteract models and the good locus}
\label[appendix]{app:algebraic-details}

We prove \Cref{prop:good-locus} using the penteract models and their
kernel maps, and conclude with a remark on the resulting \((2,2,2)\)
presentations. Our aim is to exclude curves in the fibers of the
coordinate projections. We do this by requiring certain auxiliary
determinantal surfaces to be smooth: a curve in a fiber would yield a
singularity on one of them. Throughout, we use the parameter
space and coordinate notation of \Cref{sec:algebraic}.

The two defining equations can be assembled into a single tensor, which
places our family in the penteract construction of Bhargava--Ho--Kumar
\cite{BhargavaHoKumar2016}*{Section~7}. Set
\[
\begin{aligned}
V_1&\coloneqq \rH^0(\P^1_{x,\R},\cO(1))=\Span_{\R}\{x_0,x_1\},&
V_2&\coloneqq \rH^0(\P^1_{y,\R},\cO(1))=\Span_{\R}\{y_0,y_1\},\\
V_3&\coloneqq \rH^0(\P^1_{z,\R},\cO(1))=\Span_{\R}\{z_0,z_1\},&
V_4&\coloneqq \rH^0(\P^1_{w,\R},\cO(1))=\Span_{\R}\{w_0,w_1\}.
\end{aligned}
\]
Let \(V_5\coloneqq\R^2\), with ordered basis \(e_1,e_2\) indexing the
equations. We have the identification
\[
\rH^0(P_{\R},L_{\R})
\simeq V_1\ts_{\R}V_2\ts_{\R}V_3\ts_{\R}V_4.
\]
This allows us to regard \(s\in\cV\) as the penteract
\[
s=F_{1,s}\otimes e_1+F_{2,s}\otimes e_2
\in V_1\ts_{\R}V_2\ts_{\R}V_3\ts_{\R}V_4\ts_{\R}V_5.
\]

To see how the auxiliary models arise, first eliminate the
\(w\)-coordinate from \(X_{1234,s}\coloneqq X_{s,\R}\).
For fixed \((x,y,z)\), the defining equations are two homogeneous
linear equations in \((w_0,w_1)\). Their coefficient matrix is the
contraction \(s(x,y,z,\,\cdot\,,\,\cdot\,)\), viewed as a linear map
\(V_4^\vee\to V_5\). A projective solution exists exactly when this
matrix has zero determinant. This defines the three-factor model
\[
X_{123,s}\coloneqq
\left\{(x,y,z):
\det s(x,y,z,\,\cdot\,,\,\cdot\,)=0\right\}
\subset \P(V_1^\vee)\times\P(V_2^\vee)\times\P(V_3^\vee).
\]
Each matrix entry has multidegree \((1,1,1)\), so the determinant has
multidegree \((2,2,2)\). Forgetting \(w\) gives a morphism
\(X_{1234,s}\to X_{123,s}\). Its fiber over a point is the
projectivization of the kernel of this matrix at that point.

We now make the same construction for any choice of factors.
For a four-element subset \(I\subset\{1,\ldots,5\}\), let \(m\) be the
omitted index and define
\[
X_{I,s}
\coloneqq
\left\{(v_r)_{r\in I}:
s\bigl((v_r)_{r\in I},\,\cdot\,\bigr)=0\right\}
\subset
\prod_{r\in I}\P(V_r^\vee).
\]
Here the \(v_r\) are nonzero representatives in \(V_r^\vee\).
The contraction is an element of \(V_m\), so after choosing a basis of
\(V_m\), the model \(X_{I,s}\) is cut out by two equations of multidegree
\((1,1,1,1)\). This recovers the model \(X_{1234,s}\) above.

Let \(J\) be a three-element subset of \(\{1,\ldots,5\}\), and let
\(l,m\) be the two omitted indices. Contracting \(s\) at
\((v_r)_{r\in J}\) gives a \(2\times2\) matrix, or equivalently a linear
map \(V_l^\vee\to V_m\). Bhargava--Ho--Kumar define
\[
X_{J,s}
\coloneqq
\left\{(v_r)_{r\in J}:
\det s\bigl((v_r)_{r\in J},\,\cdot\,,\,\cdot\,\bigr)=0\right\}
\subset
\prod_{r\in J}\P(V_r^\vee).
\]
This is a determinantal hypersurface of multidegree \((2,2,2)\)
\cite{BhargavaHoKumar2016}*{Section~7.1}.

Suppose that \(J=\{i,j,k\}\subset I=\{i,j,k,l\}\), and let \(m\) be
the remaining index. Forgetting the \(l\)-coordinate gives a morphism
\[
q_{I,J}\colon X_{I,s}\longrightarrow X_{J,s}.
\]
Its fiber over \((v_i,v_j,v_k)\) is the projectivization of the kernel of
the linear map
\[
s(v_i,v_j,v_k,\,\cdot\,,\,\cdot\,)\colon V_l^\vee\longrightarrow V_m.
\]
Thus the fiber is a single point when the matrix has rank one, and the
whole \(\P(V_l^\vee)\) when the matrix is zero. At a zero contraction,
the determinant and all its first derivatives vanish. If \(X_{J,s}\)
is a surface, such a point is therefore singular. Bhargava--Ho--Kumar
call it a \emph{rank singularity}
\cite{BhargavaHoKumar2016}*{Section~7.1}.

If \(X_{J,s}\) is a geometrically smooth surface, the contraction matrix
has rank one along \(X_{J,s}\). The projectivization of its kernel varies
algebraically and gives an inverse to \(q_{I,J}\), defined over \(\R\). Thus \(q_{I,J}\) is an
isomorphism whenever the three-factor model is geometrically smooth.

For distinct \(i,j\in\{1,2,3,4\}\), denote the coordinate projection by
\[
p_{ij}\colon X_{1234,s}\longrightarrow
\P(V_i^\vee)\times\P(V_j^\vee).
\]
If \(\{i,j,k,l\}=\{1,2,3,4\}\), the map factors through either
corresponding determinantal model. For example,
\[
p_{ij}\colon X_{1234,s}\longrightarrow X_{ijk,s}
\longrightarrow \P(V_i^\vee)\times\P(V_j^\vee).
\]
In particular,
\[
p_{24}=\pi_{yw},
\qquad
p_{13}=\pi_{xz}.
\]
Thus the coordinate projections in \eqref{eq:two-projections} are among
the natural projections attached to the penteract models.

Bhargava--Ho--Kumar note that all ten three-factor models are nonsingular
for a general penteract
\cite{BhargavaHoKumar2016}*{paragraph following Theorem~8.1}.
The kernel maps then identify each four-factor model with a smooth
three-factor model. Hence all fifteen models are smooth surfaces for a
general penteract. We use this common smooth locus to prove the
proposition. Requiring the auxiliary models to be smooth will rule out
curves in the fibers of the coordinate projections: each possible curve
would produce a rank singularity in one of the three-factor models.
The remaining fibers have length two, giving the desired double covers.

\begin{proof}[Proof of \Cref{prop:good-locus}]
Define
\begin{equation}\label{eq:penteract-good-locus}
\cG\coloneqq
\bigcap_{\substack{I\subset\{1,\ldots,5\}\\
\abs{I}\in\{3,4\}}}
\left\{s\in\cV:
X_{I,s}\text{ is a geometrically smooth surface}\right\}.
\end{equation}
The Jacobian criterion and properness of the ambient products show that
these conditions define a Zariski-open subscheme of
\(\mathbb A^{32}_{\R}\). Its complexification is nonempty by the preceding
generic smoothness statement. Since \(\R^{32}\) is Zariski dense in
\(\mathbb A^{32}_{\R}\), its real locus \(\cG\) is nonempty as well.
For \(s\in\cG\), the surface \(X_{s,\R}=X_{1234,s}\) is smooth by
definition, proving~\textup{(i)}.

To prove~\textup{(ii)}, we first show that every coordinate projection
has zero-dimensional geometric fibers by excluding a common curve in
its two defining equations. Fix distinct \(i,j\in\{1,2,3,4\}\), and write
\(\{i,j,k,l\}=\{1,2,3,4\}\). Over a geometric point \((v_i,v_j)\),
the fiber of \(p_{ij}\) is cut out on
\(\P(V_k^\vee)\times\P(V_l^\vee)\) by two \((1,1)\)-forms
\(g_0,g_1\). A positive-dimensional fiber would give a common divisor
of bidegree \((1,0)\), \((0,1)\), or \((1,1)\).
A common \((1,0)\)-factor gives a point \(v_k\) at which
\[
s(v_i,v_j,v_k,\,\cdot\,,\,\cdot\,)=0.
\]
This point is a rank singularity of \(X_{ijk,s}\).
A common \((0,1)\)-factor gives a rank singularity of
\(X_{ijl,s}\). A common \((1,1)\)-factor makes \(g_0,g_1\) linearly
dependent, so there is a \(v_5\in\P(V_5^\vee)\) for which
\(s(v_i,v_j,\,\cdot\,,\,\cdot\,,v_5)\) is the zero matrix. This gives
a rank singularity of \(X_{ij5,s}\). All three cases contradict
\(s\in\cG\).

Thus \(p_{ij}\) is projective with zero-dimensional geometric fibers,
and hence finite. Each fiber is a proper intersection of two
\((1,1)\)-divisors, so its scheme-theoretic length is calculated by the intersection number
\[
(1,1)\mathbin{\cdot}(1,1)=2.
\]
The source is smooth, hence Cohen--Macaulay, and the target is a regular
surface. Miracle flatness therefore makes \(p_{ij}\) finite flat of
degree two \cite{Matsumura1986}*{Theorem~23.1}.
Taking \((i,j)=(2,4)\) and \((i,j)=(1,3)\) gives the two projections
in \eqref{eq:two-projections}.
\end{proof}

In the terminology of Bhargava--Ho--Kumar, a penteract
is \emph{nondegenerate} if all five four-factor models are smooth
\cite{BhargavaHoKumar2016}*{Section~7.1, pp.~33--34}.
In \eqref{eq:penteract-good-locus} we also require the ten three-factor
models to be smooth, so all the kernel maps are isomorphisms on \(\cG\).

\begin{remark}[A \((2,2,2)\) presentation]
\label[remark]{subsec:relation-222}
For every \(s\in\cG\), the complete intersection \(X_{s,\R}\) of two
\((1,1,1,1)\)-divisors is isomorphic over \(\R\) to the smooth
determinantal \((2,2,2)\) surface
\[
Y_{s,\R}\coloneqq X_{123,s}
=V_{(\P^1_{\R})^3}\bigl(\det s(x,y,z,\,\cdot\,,\,\cdot\,)\bigr).
\]
The isomorphism is the kernel map \(q_{1234,123}\), which forgets the
\(w\)-coordinate. Its inverse adds the unique point in the projectivization
of the kernel of the contraction matrix. After complexification, this identifies
\(X_s\) with a smooth \((2,2,2)\)-K3 surface
\cite{BhargavaHoKumar2016}*{Section~7.1, pp.~33--34}.
\end{remark}

\bibliographystyle{amsalpha}
\bibliography{refs}

\bigskip
\footnotesize

\textsc{Courant Institute of Mathematical Sciences, 
New York University, 251
Mercer St, New York, NY 10012
}
\par\nopagebreak
\textit{Email address}, \texttt{junyu.cao@nyu.edu
}\par\nopagebreak
\textit{Homepage}, \url{https://junyucao1024.github.io}.

\end{document}